\documentclass[reqno,english]{amsart}
\usepackage{amsfonts,amsmath,latexsym,verbatim,amscd,mathrsfs,color,array}
\usepackage[colorlinks=true]{hyperref}

\hypersetup{
	colorlinks=true,
	linkcolor=red
}

\allowdisplaybreaks

\usepackage{float}
\usepackage{amsmath,amssymb,amsthm,amsfonts,graphicx,color,mathtools}
\usepackage[hmargin=2.5cm, vmargin=2.5cm]{geometry}
\usepackage{bbm}
\usepackage{amssymb}
\usepackage{pdfsync}
\usepackage{epstopdf}
\usepackage{cite}
\usepackage{graphicx}
\usepackage{multirow}
\usepackage[font=bf,aboveskip=15pt]{caption}
\usepackage[toc,page]{appendix}
\usepackage[colorlinks=true]{hyperref}
\allowdisplaybreaks[4]

\newcommand{\1}{\mathbf{1}}

\newcommand{\TT}{{\mathcal T}  }

\newcommand{\pp}{ {\partial} }

\newcommand{\RR}{{{\mathbb R}}}

\newcommand{\R} {\mathbb R}

\newcommand{\cuad}{{\sqcap\kern-.68em\sqcup}}

\newcommand{\foral}{\quad\mbox{for all}\quad}

\newcommand{\be}{\begin{equation}}
\newcommand{\ee}{\end{equation}}

\newcommand{\la}{\lambda}

\newtheorem{lemma}{Lemma}[section]
\newtheorem{proposition}{Proposition}[section]
\newtheorem{theorem}{Theorem}[section]

\newtheorem{remark}{Remark}[section]
\newcommand{\bremark}{\begin{remark} \em}
	\newcommand{\eremark}{\end{remark} }

\numberwithin{equation}{section}
\begin{document}
\title[Global solutions for the Yang-Mills heat flow]{Global oscillatory  solutions for the Yang-Mills heat flow}

\author[Y. Sire]{Yannick Sire}
\address{\noindent Department of Mathematics, Johns Hopkins University, 404 Krieger Hall, 3400 N. Charles Street, Baltimore, MD 21218, USA}
\email{ysire1@jhu.edu}
\author[J. Wei]{Juncheng Wei}
\address{\noindent
Department of Mathematics,
Chinese University of Hong Kong,
Shatin, NT, Hong Kong}
\email{wei@math.cuhk.edu.hk}

\author[Y. Zheng]{Youquan Zheng}
\address{\noindent School of Mathematics, Tianjin University, Tianjin 300072, P. R. China}
\email{zhengyq@tju.edu.cn}

\author[Y. Zhou]{Yifu Zhou}
\address{\noindent
School of Mathematics and Statistics, Wuhan University, Wuhan 430072, China}
\email{yifuzhou@whu.edu.cn}

\maketitle

\begin{abstract}
We investigate the long-time dynamics for the global solution of the $SO(4)$-equivariant Yang-Mills heat flow (YMHF) with structure group $SU(2)$ in space dimension $4$. For a class of initial data with specific decay at spatial infinity, we prove that the long-time dynamics of YMHF can be described by the initial data in a unified manner. As a consequence, the global solutions can exhibit blow-up, blow-down, and more exotically, {\it oscillatory} asymptotic behavior at time infinity. This seems to be the first example of Yang-Mills heat flows with oscillatory behavior as $t\to \infty$.  
\end{abstract}

\bigskip

\section{Introduction}
{\bf The Yang-Mills equation and its heat flow}. We investigate the long time behaviour of an equivariant version of the Yang-Mills heat flow in four dimensions and exhibit a wide range of phenomena depending on the initial condition. 

Let $E\to M$ be a vector bundle over a four dimensional Riemannian manifold without boundary, with its  structure group $G$ being a compact Lie group.  A connection $A$ on $E$ can be specified as a covariant derivative $D_A$ from $C^\infty(E)$ to $C^\infty (E\otimes \Omega^1(M))$. In a local trivialization of the vector bundle $E$, the covariant derivative $D_A$ can be expressed as 
$$
D:=D_A=d+A_\alpha
$$
where $A=(A_\alpha)_{\alpha}$ is a section of $T^*M \otimes \mathfrak{g}$, and $\mathfrak{g}$ is the Lie algebra of $G$, i.e. the  connection $A$ is a $\mathfrak{g}-$valued $1-$form.  The curvature $F_A$ of the connection $A$ is given by the tensor $D_A^2: \Omega^0(E)\to \Omega^2(E)$, which can be formally  written as
$$
F_A:=dA+A\wedge A.
$$
For a connection $A$, the Yang-Mills functional is defined by
$$
YM(A) = \frac{1}{2}\int_{M}|F_A|^2 dx.
$$
It is well known that the Euler-Lagrange equation of $YM(A)$ is
\begin{equation}\label{YMstat}
D^*_AF_A=0
\end{equation}
where $D^*_A$ denotes the adjoint operator of $D_A$ with respect to the Killing form of $G$ and the metric on $M$. A connection $A$ is called Yang-Mills if and only if it is a critical point of $YM$, which is equivalent to the equation
$$
D^*_AF_A = 0.
$$
There is a substantial literature related to solutions of this equation and we do not attempt to be exhaustive but just mention few classical references such as the book by Donaldson and Kronheimer \cite{DonaldsonKronheimer} or the one by Freed and Uhlenbeck \cite{Freedbook}. The importance of gauge theory and Equation \eqref{YMstat} in topology and geometry has been emphasized in e.g. \cite{DonaldsonKronheimer,TianGangGTCG,UY} (without being exhaustive at all). An important aspect of the solutions to Yang-Mills equations is their regularity and the structure of their possible singularities (see e.g. the important works \cite{MeyerRiviereRMI}, \cite{NaberValtortaInvenMath}, \cite{Tao-TianJAMS}, \cite{TianGangGTCG}, \cite{UhlenbeckCMP-I}, \cite{UhlenbeckCMP-II}). 

To construct Yang-Mills connections on any given bundle $E$, a natural approach is to deform a given connection along the negative gradient flow of $YM$ which can be described by the evolution equation
\begin{equation}\label{e:Yangmillsheatflow}
\frac{\partial A}{\partial t} = -D_A^*F_A,
\end{equation}
starting from any initial connection $A_0$. As for its static version there has been an important amount of work to understand regularity {\sl vs} singularity formation, and the long time behaviour of suitable solutions. See e.g. \cite{feehan,WaldronInventionMath,SchlatterStruweAJM,StruweCVPDE,Schlatter2,SchlatterCrelle,naito} without being exhaustive as well. 

{\bf The BPST/ADHM instantons}. Let us recall the well known BPST/ADHM instantons for the Yang-Mills equation with structure group $SU(2)$, see \cite{AHDM}.
Identify the quaternion
$$
x = x_1 + x_2i + x_3 j + x_4k\in \mathbb H
$$
with elements of $\mathbb{R}^4$. $i$, $j$, $k$ satisfy the following algebraic property
$$
i^2 = j^2 = k^2 = -1,\quad ij = k = -ji,\quad jk = i = -kj,\quad ki = j = -ik,
$$
and
$$
\bar{x} = x_1 -x_2 i - x_3 j - x_4 k, \quad |x|^2 = x_1^2 + x_2^2 + x_3^2 + x_4^2 = x\cdot\bar{x}, \quad Im x = x_2i + x_3 j + x_4k
$$
hold. It is well known that there is an isomorphism between the Lie algebra $su(2)$ of structure group $SU(2)$ and $Im \mathbb H$. Assume that $B(x) = Im(f(x, \bar{x})d\bar{x})$, it was pointed out by Polyakov that, when $f(x, \bar{x}) = \frac{x}{1+|x|^2}$, it is a nontrivial self-dual instanton on the principal bundle $E = \mathbb{R}^4\times SU(2)$. In this case, one has
$$
B = Im\left(\frac{x}{1+|x|^2}d\bar{x}\right), \quad F_B = \frac{dx\wedge d\bar{x}}{(1+|x|^2)^2}.
$$
Note that
\begin{equation*}
\begin{aligned}
dx\wedge d\bar{x} &= \left(dx_1 + i dx_2 + j dx_3 + kdx_4\right)\wedge \left(dx_1 - i dx_2 - j dx_3 - kdx_4\right)\\
& = -2\left[i(dx_1\wedge dx_2 + dx_3\wedge d_4) + j(dx_1\wedge dx_3 + dx_4\wedge d_2)\right. \\
&\quad \left. + k(dx_1\wedge dx_4 + dx_2\wedge d_3)\right]
\end{aligned}
\end{equation*}
form a basis for a self-dual 2-form. See also \cite{AHDM}, \cite{BPST} and \cite{Donaldson} for more explicit constructing methods for Yang-Mills instantons.

{\bf The $SO(4)$-equivariant case}. We restrict ourselves to $SO(4)$-equivariant solutions of (\ref{e:Yangmillsheatflow}) with structure group $SU(2)$, which means that assuming $A(x, t) = Im(\frac{x}{2r^2}\psi(r, t)d\bar{x})$, $r = |x|$, then the Yang-Mills heat flow takes the following form,
\begin{equation}\label{YMSO4}
\psi_t= \psi_{rr} + \frac{1}{r}\psi_r -\frac{2}{r^2}(\psi-1)(\psi-2)\psi,
\end{equation}
see, for example, \cite{Grotowski-Shatah} and \cite{SchlatterStruweAJM}.
This equation is the formal negative $L^2$-gradient flow associated to the reduced energy functional
$$
F[\psi]=\int_0^\infty \left[\psi_r^2+\frac{4}{r^2}\psi^2\left(1-\frac{\psi}{2}\right)^2\right] r dr
$$
and 
possesses a one-parameter family of finite energy, static solutions, namely
\begin{equation*}
\frac{2r^2}{r^2+\lambda^2},\quad \lambda\in\mathbb{R}_+.
\end{equation*}
If we make the transformation $\phi = r^{-2}\psi$, then (\ref{YMSO4}) becomes the following heat equation in dimension 6 under radial symmetry
\begin{equation}\label{YMSO4-6Dheat}
u_t = u_{rr} + \frac{5}{r}u_r +(6-2r^2 u)u^2.
\end{equation}

{\bf Long time behavior of Yang-Mills heat flow.}  Struwe \cite{StruweCVPDE} (see also \cite{Schlatter2}) established global existence and uniqueness for the Yang-Mills heat flow in a vector bundle over a compact Riemannian four dimensional manifold for given initial connection of finite energy. It was proved in \cite{SchlatterStruweAJM} that the Yang-Mills heat flow of $SO(4)$-equivariant connection on an $SU(2)$-bundle over a ball in $\mathbb{R}^4$ admits a smooth solution for all times and  for any initial and boundary data of finite energy. In \cite{Hong-TianCAG}, the global existence of smooth solutions to the $m$-equivariant Yang-Mills flow on $\mathbb{R}^4$ as well as the Yang-Mills-Higgs flow on a complete three-manifold were obtained. In \cite{WaldronInventionMath}, the author proved that finite-time singularities do not occur in Yang-Mills flow on a compact four-manifold, confirming the conjecture of \cite{SchlatterStruweAJM}. In \cite{SWZ-YMHF}, the authors constructed some blow-up solutions at infinity for the 4D Yang-Mills. This result suggests that the asymptotic behaviour of Yang-Mills flow can be pretty intricate.   

{\bf Results: global oscillatory, unbounded and decaying solutions to the Yang-Mills heat flow}. In this paper, we prove that global solutions for the Yang-Mills heat flow can exhibit growing, decaying and more exotically, {\it oscillating} long-time behaviors.
\begin{theorem}\label{thm}
Let $\Theta(r,t)$ solve 
\begin{equation}
    \left\{
\begin{aligned}
    &\Theta_t=\Theta_{rr}+\frac5r \Theta_r, ~&(r,t)\in \R_+\times (t_0,\infty),\\
    &\Theta(r,t_0)=\Theta_0(r),~& r\in \R_+.
\end{aligned}
\right.
\end{equation}
Then for any smooth initial data $\Theta_0(r)$ satisfying 
$$
|\Theta_0(r)|\sim r^{-2}(\log r)^{-a} ~\mbox{ as }~r\to+\infty,\quad 0< a <1,
$$
problem (\ref{YMSO4-6Dheat}) has a solution with the following form
$$
u(r,t)= \la^{-2}(t)U\left(\frac{r}{\la(t)}\right)+\Theta(r,t)+\Psi_*(r,t),
$$
where $U(\rho) = \frac{2}{\rho^2+1}$, and $\|\Psi_*(\cdot,t)\|_{L^\infty}=O(t^{-1}(\log t)^{-a-\epsilon})$ for some $0<\epsilon<1$. 
The global scaling law of $\lambda$ is described precisely by
\begin{equation}\label{law}
   \log\la(t) =  \left(-\frac{3}{2}+O\left(\frac{1}{\log t}\right)\right)\int_{\sqrt{t_0}}^{\sqrt{t}} s\Theta_0(s)ds + C(t_0, t).
\end{equation}   
Here $C(t_0, t)$ is a smooth bounded function of $t\in (t_0, \infty)$.
\end{theorem}

\begin{remark}
Theorem \ref{thm} immediately yields the existence of the following three regimes: blow-up, blow-down, and exotic oscillation.
\begin{enumerate}
\item If the initial data $\Theta_0$ is positive and $\Theta_0(r)\sim\langle r\rangle^{-2}(\log\langle r\rangle)^{-a}$, then there exists positive global unbounded solution with $$\|u(\cdot,t)\|_{L^\infty}\sim e^{c(\log t)^{1-a}} ~\mbox{ as }~t\to+\infty$$
for some $c>0$.

\item If the initial data $\Theta_0$ is negative and $\Theta_0(r)\sim-\langle r\rangle^{-2}(\log\langle r\rangle)^{-a}$, then there exists a sign-changing blow-down solution with $$\|u(\cdot,t)\|_{L^\infty}\sim e^{-c(\log t)^{1-a}} ~\mbox{ as }~t\to +\infty$$  for some $c>0$.

\item 
There exists a sign-changing smooth function $\Theta_0$ with $$|\Theta_0(r)|\sim r^{-2}\log^{-a} r ~\mbox{ as }~r\to+\infty,\quad 0<a<1$$
such that
$$ 
\liminf_{t\to +\infty}\lambda(t) = 0\quad \text{and}\quad \limsup_{t\to +\infty}\lambda(t) = +\infty.
$$
A particular choice can be 
\begin{align*}
\Theta_0(r) = \frac{(1-a)\cos(\log\log(2+r^2))-\sin(\log\log(2+r^2))}{(2+r^2)\left[\log(2+r^2)\right]^a},\quad 0<a<1.
\end{align*}
\end{enumerate}
\end{remark}

\begin{remark}
    A similar construction as in the proof of Theorem \ref{thm} (and also \cite{weizhouEPE}) implies that there exists initial data $\Theta_0(r)$ that decays faster than the soliton as 
$$
|\Theta_0(r)|\sim r^{-2-\epsilon} ~\mbox{ as }~r\to \infty
$$
for some $\epsilon>0$, such that the global solution of \eqref{YMSO4-6Dheat} remains bounded 
$$
\|u(\cdot,t)\|_{L^\infty}\sim 1 ~\mbox{ as }~t\to \infty.
$$
\end{remark}

As far as the long time behavior is concerned, the deep relation between degree 2 harmonic map heat flow, i.e.
$$
u_t=u_{rr}+\frac1r u_r-\frac{4\sin u\cos u}{r^2}, \quad (r,t)\in \R_+\times\R_+,
$$
and the four-dimensional  Yang-Mills heat flow has long been known (see, for example, \cite{Grotowski-Shatah}).
In this regard, Theorem \ref{thm} might be viewed as near-soliton global dynamics analogous to the result by Gustafson-Nakanishi-Tsai \cite{GNT10CMP}, where they investigated exotic asymptotic behaviors of 2-equivariant harmonic map heat flow and classified the global scaling law as
$$\log\la(t)\sim \frac{2}{\pi}\int_1^{\sqrt t}\frac{v_1(s)}{s}ds,
$$
only depending on the first entry of the initial map of the equivariant flow. The current study is motivated by the work \cite{GNT10CMP}, while the method employed here, in the spirit of \cite{CortazarDelPinoMussoJEMS,DavilaDelPinoWei2020}, is completely different. It involves gluing two well matched inner and outer pieces and finding desired solutions in refined weighted spaces, requiring the development of delicate linear theory in the $L^\infty$-framework. The global scaling law \eqref{law} is obtained by a careful treatment of the orthogonality condition, ensuring sufficiently fast pointwise decay for the inner solution.

On the other hand, another heat flow that shares a similar structure is the six-dimensional  Fujita equation with critical nonlinearity
\begin{equation}\label{6DFujita}
u_t=\Delta u+u^2 ~\mbox{ in }~\R^6\times \R_+,
\end{equation}
and the current study is also motivated by a program proposed by Fila and King \cite{FilaKing12} concerning the threshold behavior of global solutions depending on the initial data in a precisely manner. The Fila-King program has received much attention in recent years, starting from the first rigorous construction  \cite{173D} by del Pino-Musso-Wei. Very recently, Harada \cite{HaradaOsc} achieved rather refined constructions of global unbounded, decaying and oscillatory solutions for \eqref{6DFujita}; see also  \cite{weizhouEPE} for an earlier result on the existence of positive  bounded solutions. 

To prove Theorem \ref{thm}, we use the parabolic gluing method developed in \cite{CortazarDelPinoMussoJEMS} and \cite{DavilaDelPinoWei2020}. In Section \ref{gluing-scheme}, we set up the inner-outer scheme and give the linear theory for the inner problem and outer problem, respectively.  In Section \ref{outer-problem}, we give the estimates for the outer problem. To ensure the solvable of the inner problem, we solve the scaling parameter in Section \ref{oscillating-scale}. In this process, from the orthogonality condition, we recover the global scaling law of Gustafson-Nakanishi-Tsai \cite[Theorem 1.2]{GNT10CMP} under the assumption that $\Theta_0(r)\sim r^{-2}\log^{-a} r$ as $r\to+\infty$, $0<a<1$, see Section \ref{GNT-formula}. In Section \ref{inner-problem}, we solve inner problem and complete the proof of Theorem \ref{thm}. 

\medskip

\section{The inner outer gluing scheme and estimates for the linear problems}\label{gluing-scheme}
The Yang-Mills heat flow in four dimensions reduces to 
\begin{equation}\label{eqn-YMHF}
\left\{
\begin{aligned}
    &u_t=u_{rr}+\frac5r u_r+6u^2-2r^2 u^3, ~&(r,t)\in \R_+\times \R_+,\\
    &u(r,t_0)=u_0(r),~& r\in \R_+
\end{aligned}
\right.
\end{equation}
with symmetry. It admits one-parameter family of steady solutions
\begin{equation}\label{YM-soliton}
    U_\la(r)=\frac{2}{r^2+\la^2}:=\la^{-2} U(\rho),\quad \rho:=\frac{r}{\la}\foral \la>0,
\end{equation}
known as the gravitational instantons. We denote the infinitesimal generator under dilation by 
\begin{equation}\label{kernel}
    Z(\rho)=\frac1{(1+\rho^2)^2}.
\end{equation}
Using \eqref{YM-soliton}, we aim to construct global solutions that exhibit precise long-term asymptotics.
Take the ansatz
\begin{equation}
    u=U_{\la(t)}(r)+\Theta(r,t)+\eta_R \lambda^{-2}\phi(\rho,t)+\psi(r,t)
\end{equation}
with
\begin{equation}\label{eqn-Theta}
    \left\{
\begin{aligned}
    &\Theta_t=\Theta_{rr}+\frac5r \Theta_r, ~&(r,t)\in \R_+\times (t_0,\infty),\\
    &\Theta(r,t_0)=\Theta_0(r),~& r\in \R_+
\end{aligned}
\right.
\end{equation}
for some initial data $\Theta_0(r)$ to be constructed later. Here $\eta_R:=\eta(\rho/R)$ with $\eta$ being the smooth cut-off function with $\eta(x) = 1$ for $|x|\leq 1$, $\eta(x) = 0$ for $|x|\geq 2$, and $R$ is a sufficiently large number.
Then we compute
\begin{equation}
    \begin{aligned}
    &~\psi_t+\lambda^{-2}\pp_t\eta_R \phi+\lambda^{-2}\eta_R(\phi_t-2\la^{-1}\dot\la\phi-\rho\phi_\rho \la^{-1}\dot\la)\\&~=\psi_{rr}+\frac5r{\psi_r}+\eta_R\la^{-2}\left(\phi_{rr}+\frac5r\phi_r\right)+2\la^{-2}\pp_r\eta_R\phi_r+\la^{-2}\phi\left(\pp_{rr}\eta_R+\frac5r\pp_r\eta_R\right)\\
        &~\quad  +6(U_\la+\Theta+\eta_R\la^{-2}\phi+\psi)^2-2r^2(U_\la+\Theta+\eta_R\la^{-2}\phi+\psi)^3+\frac{4\dot\la \la^{-3}}{(1+\rho^2)^2}-6U_\la^2+2r^2 U_\la^3
    \end{aligned}
\end{equation}
and split
\begin{equation}\label{eqn-inner}
    \la^2 \phi_t=\phi_{\rho\rho}+\frac5\rho\phi_\rho+(12U-6\rho^2 U^2)\phi+\la^2(12U-6\rho^2 U^2)(\psi+\Theta)+\frac{4\la\dot\la}{(1+\rho^2)^2} + \lambda^2\mathcal N[\la,\phi,\psi]~\mbox{ in }~B_{2R}\times (t_0,\infty),
\end{equation}
\begin{equation}\label{eqn-outer}
    \begin{aligned}
\psi_t=&~\psi_{rr}+\frac5r\psi_r+\la^{-2}(1-\eta_R)(12U-6\rho^2 U^2)(\psi+\Theta)+2\la^{-2}\pp_r\eta_R\phi_r+\la^{-2}\phi\left(\pp_{rr}\eta_R+\frac5r\pp_r\eta_R\right)\\
    &~+2\la^{-3}\dot\la \eta_R\phi+\la^{-3}\dot\la \eta_R \rho\phi_\rho+(1-\eta_R)\frac{4\dot\la \la^{-3}}{(1+\rho^2)^2}+\mathcal N[\la,\phi,\psi] ~\mbox{ in }~\R_+\times (t_0,\infty),
    \end{aligned}
\end{equation}
where
\begin{equation}\label{def-N}
    \begin{aligned}
        \mathcal N[\la,\phi,\psi]:=&~6(U_\la+\Theta+\eta_R\la^{-2}\phi+\psi)^2-2r^2(U_\la+\Theta+\eta_R\la^{-2}\phi+\psi)^3-6U_\la^2+2r^2 U_\la^3\\
        &~-(12U_\la-6r^2 U_\la^2)(\Theta+\eta_R \la^{-2}\phi+\psi).
    \end{aligned}
\end{equation}
The equations \eqref{eqn-inner} and \eqref{eqn-outer} are called respectively inner problem and outer problem and will be solved under zero initial condition. The estimates for the outer problem \eqref{eqn-outer} will be done by the Duhamel's formula in 6D with radial symmetry, and the resolution of the inner problem \eqref{eqn-inner} will be achieved by {\it apriori} estimates of the linear problem 
\begin{equation}\label{inner-linear-eq}
    \left\{
    \begin{aligned}
        &~\pp_\tau\phi=\phi_{\rho\rho}+\frac5\rho\phi_\rho+(12U-6\rho^2 U^2)\phi+H(\rho,\tau) ~\mbox{ in }~B_{2R}\times (\tau_0,\infty)\\
        &~\phi(\rho,\tau_0)=0 ~\mbox{ in }~B_{2R},
    \end{aligned}
    \right.
\end{equation}
where
$$
\tau=\tau(t):=\int_{t_0}^t \la^{-2}(s)ds +C t_0\la^{-2}(t_0),\quad \tau_0:=\tau(t_0)
$$
for sufficiently large $C>0$. We will use  the norms 
\begin{equation*}
\|H\|_{*}:=\sup_{\tau>\tau_0, ~\rho\in B_{2R}}\tau^\nu\langle \rho\rangle^{2+a}|H(\rho,\tau)|,\quad \nu>0,~0<a<1,
\end{equation*} 
\begin{equation}\label{inner-norm}
	\|\phi\|_{\rm{in}}:=\sup_{\tau>\tau_0,~ \rho\in B_{2R}}
	\tau^{\nu}R^{a-6}(t(\tau))\langle \rho\rangle^{6}
	 \big[\langle \rho\rangle |\phi_\rho(\rho,\tau)| + |\phi(\rho,\tau)| \big] ,
\end{equation}

The linear estimates for the inner problem \eqref{eqn-inner} is established in \cite[Lemma 6.1]{SWZ-YMHF}:
\begin{lemma}\label{Prop-inner}(\cite[Lemma 6.1]{SWZ-YMHF})
   For \eqref{inner-linear-eq}, assume that $\|H\|_{*}<\infty$ and
	\begin{equation}\label{inner-orthogonal-cond}
		\int_0^{4R} H(\rho,\tau)Z(\rho)\rho^5d\rho=0, \quad \forall \tau\in(\tau_0,\infty).
	\end{equation}
	Then  for $\tau_0$ sufficiently large, there exists a solution $ \phi=\mathcal{T}_{\rm{in}}[H],
	$ linear in $H$, that satisfies the estimate
	\begin{equation*}
		\|\phi\|_{\rm{in}} \lesssim \|H\|_{*}.
	\end{equation*}
\end{lemma}

To deal with the outer problem \eqref{eqn-outer}, we define
\begin{equation*}
\mathcal{T}^{\rm{out}}_6\left[f\right]
:=
\int_{t_0}^t \int_{\R^6} 
\left[ 4\pi(t-s)\right]^{-3}
e^{ -\frac{|x-z|^2}{4(t-s)}  } f(z,s) dz ds 
\end{equation*}
under radial symmetry, and we invoke some convolution estimates in the spirit of \cite[Lemma A.1 and Lemma A.2]{4DFila-King}, the difference is that we need to deal with logarithmic decay.
\begin{lemma}\label{linear-outer}
Assume $v(t)\ge 0$, $a\in [0, 1)$, $b\in (-\infty, n]$, $n\geq 3$, $0\le l_1(t) \le l_2(t) \le C_{*} \sqrt{t}$, $C_l^{-1} l_i(t) \le l_i(s) \le C_l l_i(t)  $, $i=1,2$, for all $\frac{t}{2} \le s\le t$, $t\ge t_0\ge 0$, where $C_*>0,C_l\ge1$. Then
    \begin{equation}\label{RHS-with-upper-bound}
		\begin{aligned}
			&
			\TT_n^{out}\left[ v(t) |x|^{-b} (\log |x|)^{-a}\1_{\{ l_1(t) \le |x| \le l_2(t) \}}\right]
			\lesssim
			t^{-\frac{n}{2}}
			e^
			{-\frac{|x|^2}{16 t } }
			\int_{\frac{t_0}{2}}^{\frac{t}{2} }   v(s)
			\begin{cases}
				l_2^{n-b}(s)(\log l_2(s))^{-a}
				&
				\mbox{ \ if \ } b<n
				\\
                (\log(l_2(s)))^{1-a}
				&
				\mbox{ \ if \ } b=n
			\end{cases}
			d s
			\\
			&
			+
			\sup\limits_{ t_1 \in [t/2,t] } v(t_1)
			\begin{cases}
				\begin{cases}
					l_2^{2-b}(t)(\log l_2(t))^{-a}
					& \mbox{ \ if \ }
					b<2
					\\
				\langle	\ln (\frac{l_2(t)}{l_1(t)}) \rangle (\log l_1(t))^{-a}
					& \mbox{ \ if \ }
					b=2
					\\
					l_1^{2-b}(t)(\log l_1(t))^{-a}
					& \mbox{ \ if \ }
					b>2
				\end{cases}
				&
				\mbox{ \ for \ }  |x| \le l_1(t)
				\\
				\begin{cases}
					l_2^{2-b}(t)\langle\log l_2(t)\rangle^{-a}
					&
					\mbox{ \ if \ } b<2
					\\
				\langle\log(\frac{l_2(t)}{|x|} )  \rangle \langle\log |x|\rangle^{-a}
					&
					\mbox{ \ if \ } b=2
					\\
					|x|^{2-b}\langle\log |x|\rangle^{-a}
					&
					\mbox{ \ if \ } 2<b<n
					\\
					|x|^{2-n} \langle	\log (\frac{|x|}{l_1(t)}) \rangle\langle\log l_1(t)\rangle^{-a}
					&
					\mbox{ \ if \ } b=n
				\end{cases}
				&
				\mbox{ \ for \ }
				l_1(t) < |x| \le l_2(t)
				\\
				|x|^{2-n}
			e^{-\frac{|x|^2}{16 t}}
			\begin{cases}
				l_2^{n-b}(t)\langle	\log l_2(t)\rangle^{-a}
				&
				\mbox{ \ if \ } b<n
				\\
				\langle	\log l_2(t)\rangle^{1-a}
				&
				\mbox{ \ if \ } b=n
				\end{cases}
				&
				\mbox{ \ for \ }
				|x| > l_2(t)
			\end{cases}
		\end{aligned}
	\end{equation}
    and
    \begin{equation}\label{RHS-without-upper-bound}
		\begin{aligned}
			&
			\TT_{n}^{out}\left[v(t) |x|^{-b} (\log |x|)^{-a}\1_{\{ r\ge t^{\frac 12}\}}\right]\lesssim
			\\
			&
			\begin{cases}
				t^{-\frac{n}{2}}
				\int_{t_{0}/2}^{t/2}
				v(s)
				\begin{cases}
					t^{\frac{n-b}{2}}(\log t)^{-a}
					&
					\mbox{ \ if \ } b<n
					\\
				(\log(t))^{1-a}
					&
					\mbox{ \ if \ } b=n
				\end{cases}
				d s
				+ t^{1-\frac{b}{2}}(\log t)^{-a} \sup\limits_{ t_1 \in [t/2,t] } v(t_1)
				\mbox{ \ if \ } |x|\le t^{\frac 12}
				\\ 
				  |x|^{-b}(\log t)^{-a}
				\left( t \sup\limits_{ t_1 \in [t/2,t] } v(t_1) +	\int_{t_{0}/2}^{t/2}
				v(s) d s \right) \\
                \qquad\qquad\qquad\qquad +
				t^{-\frac{n}{2}} e^{-\frac{|x|^2}{16 t} }
				\int_{t_{0}/2}^{t/2}
				v(s)
				\begin{cases}
					0
					&
					\mbox{ \ if \ }
					b<n
					\\
				(\log(|x|))^{1-a}
				&
				\mbox{ \ if \ }
				b=n
				\end{cases}
				d s
				\mbox{ \ if \ } |x| > t^{\frac 12},
			\end{cases}
		\end{aligned}
	\end{equation}
where ``$\lesssim$'' is independent of $t_0$ and we assume $v(s)=0$ for $s<t_0$.	
\end{lemma}

Furthermore, we need to estimate the Cauchy problem with initial data. Consider the problem 
\begin{equation}\label{HeatEquation}
    \left\{
\begin{aligned}
    &\Theta_t=\Theta_{rr}+\frac5r \Theta_r, ~&(r,t)\in \R_+\times (t_0,\infty),\\
    &\Theta(r,t_0)=\Theta_0(r),~& r\in \R_+.
\end{aligned}
\right.
\end{equation}
Here, we assume $|\Theta_0(r)|\lesssim \langle r\rangle^{-2}(\log\langle r\rangle)^{-a}.$
By similar arguments as  in \cite[Lemma A.3]{4DFila-King}, we have 
\begin{lemma}\label{Lemma2.32.3}
Assume $a \in [0, 1)$ and $b < n$, $|\Theta_0(r)|\lesssim \langle r\rangle^{-b}(\log\langle r\rangle)^{-a}$, $r= |x|$, then, for $t_0$ sufficiently large, the solution $\Theta$ of (\ref{HeatEquation}) satisfies the following estimates,
\begin{equation*}
	\begin{aligned}
		&|\Theta(r,t)|
		\lesssim 
			\langle t \rangle^{-\frac{b}{2}}(\log(t+2))^{-a}\1_{\{|x|\leq\sqrt t\}} + \langle x\rangle^{-b}(\ln (|x|+2))^{-a}\1_{\{|x|\geq\sqrt t\}} ,
	\end{aligned}
\end{equation*}
\begin{align*}
&\left|\Theta_r(r,t)\right| \lesssim
			\langle t \rangle^{-\frac{b}{2}}(\log(t+2))^{-a}\frac{1}{\sqrt{t}}\1_{\{|x|\leq\sqrt t\}} + \langle x\rangle^{-b}(\ln (|x|+2))^{-a}\frac{|x|}{t}\1_{\{|x|\geq\sqrt t\}}  .
\end{align*}
In particular, for $b=2$, we have 
\begin{equation}\label{Estimate-for-Theta}
    \begin{aligned}
        &~|\Theta(r,t)|\lesssim \langle t \rangle^{-1}(\log(t+2))^{-a}\1_{\{|x|\leq\sqrt t\}}+ \langle x\rangle^{-2}(\ln (|x|+2))^{-a}\1_{\{|x|\geq\sqrt t\}},
    \end{aligned}
\end{equation}
\begin{equation}\label{Estimate-for-Theta-Gradient}
    \begin{aligned}
        &~|\Theta_r(r,t)|\lesssim \langle t \rangle^{-1}(\log(t+2))^{-a}\frac{1}{\sqrt{t}}\1_{\{|x|\leq\sqrt t\}}   + \langle x\rangle^{-2}(\ln (|x|+2))^{-a}\frac{|x|}{t}\1_{\{|x|\geq\sqrt t\}}
    \end{aligned}
\end{equation}
 for $t\geq t_0$.
\end{lemma}
We postpone the proof of Lemma \ref{linear-outer} and Lemma \ref{Lemma2.32.3} to Appendix \ref{Appendix-convo}.

\medskip

\section{Global Scaling Law}\label{GNT-formula}

\medskip

In this section, we aim to capture the main global scaling law for $\la(t)$ with precise dependence on the initial data $\Theta_0(r)$ in the spirit of  \cite[Theorem 1.2]{GNT10CMP}. We will show that this can be achieved by the orthogonality condition \eqref{inner-orthogonal-cond}. Indeed, the two main terms in \eqref{eqn-inner} enter \eqref{inner-orthogonal-cond} as
\begin{equation}\label{orthog-214214}
    \int_0^{4R} \left[(12U-6\rho^2 U^2)\Theta(\la\rho,t)+\frac{4\la^{-1}\dot\la}{(1+\rho^2)^2}\right]Z(\rho)\rho^5 d\rho\approx 0.
\end{equation}
Here, the solution $\Theta(r,t)=\Theta(\la\rho,t)$ of (\ref{eqn-Theta}) can be obtained using radial heat kerne as follows,
\begin{equation*}
    \begin{aligned}
    &~\Theta(r,t)=\int_0^\infty \Gamma_6(r,s;t)\Theta_0(s)ds\\
    &~\Gamma_6(r,s;t):=\frac{s^3}{2t r^2}e^{-\frac{r^2+s^2}{4t}}I_2(\frac{rs}{2t}),\quad I_2(\xi)\sim\begin{cases}
        \frac{\xi^2}{8}(1+O(\xi^2)),\quad &\xi\to 0\\
        \frac{1}{\sqrt{2\pi}}e^\xi \xi^{-1/2}(1+O(\frac{1}{\xi})),\quad &\xi\to\infty
    \end{cases}.
    \end{aligned}
\end{equation*}
The asymptotics above can be found in \cite[pp. 375-377]{AbramowitzStegun}. We postpone the derivation of radial heat kernel in Appendix \ref{radial-heat-kernel}. Then we have the following characterization of the scaling parameter. 
\begin{proposition}
Suppose $\la R\ll\sqrt{t}$ and the initial data $\Theta_0(r)$ decays like
$$
\Theta_0(r)\sim r^{-2}\log^{-a} r ~\mbox{ as }~r\to+\infty,\quad 0<a<1,
$$
then the solution of (\ref{orthog-214214}) satisfies the following relation,
\begin{equation}\label{GNK-law}
\log\la(t) =  \left(-\frac{3}{2}+O\left(\frac{1}{\log t}\right)\right)\int_{\sqrt{t_0}}^{\sqrt{t}} s\Theta_0(s)ds + C(t_0, t).
\end{equation}    
Here $C(t_0, t)$ is a smooth bounded function of $t\in (t_0, \infty)$.
\end{proposition}
\begin{proof}
Recall we have $\la R\ll\sqrt{t}$, then for the first term in (\ref{orthog-214214}),
\begin{align}\label{def-I1234}
&~\int_0^{4R} (12U-6\rho^2 U^2)\Theta(\la\rho,t)Z(\rho)\rho^5 d\rho\notag\\
&~ =\int_0^{4R} V(\rho)Z(\rho)\rho^5 d\rho\int_0^\infty \frac{s^3}{2t\la^2\rho^2}e^{-\frac{\la^2\rho^2+s^2}{4t}}I_2(\frac{\la\rho s}{2t})\Theta_0(s)ds\\&~=\frac1{64t^3}\int_0^{4R}V(\rho)Z(\rho)\rho^5 e^{-\frac{\la^2\rho^2}{4t}}d\rho\int_0^{\frac{2t}{\la\rho}} \Theta_0(s) s^5 e^{-\frac{s^2}{4t}}ds\notag\\
&~\quad +\frac{1}{\sqrt{\pi}}\int_0^{4R}V(\rho)Z(\rho)\rho^5\frac{1}{(\la\rho)^{5/2}t^{1/2}}d\rho\int_{\frac{2t}{\la\rho}}^\infty s^{5/2}\Theta_0(s) e^{-\frac{(\la\rho-s)^2}{4t}}ds\notag\\
&~=\int_0^{\sqrt t}\frac{s^5}{64t^3} e^{-\frac{s^2}{4t}}\Theta_0(s)ds\int_0^{4R}V(\rho) Z(\rho)\rho^5 e^{-\frac{\la^2\rho^2}{4t}} d\rho\notag\\
&~\quad+\int_{\sqrt t}^{\frac{t}{2\la R}}\frac{s^5}{64t^3} e^{-\frac{s^2}{4t}}\Theta_0(s)ds\int_0^{4R}V(\rho) Z(\rho)\rho^5 e^{-\frac{\la^2\rho^2}{4t}} d\rho\notag\\
&~\quad+\int_{\frac{t}{2\la R}}^\infty \frac{s^5}{64t^3} e^{-\frac{s^2}{4t}}\Theta_0(s)ds \int_0^{\frac{2t}{\la s}} V(\rho) Z(\rho)\rho^5 e^{-\frac{\la^2\rho^2}{4t}} d\rho\notag\\
&~\quad+\la^{-5/2}\frac{1}{\sqrt{\pi t}}\int_{\frac{t}{2\la R}}^\infty s^{5/2}\Theta_0(s) ds \int_{\frac{2t}{\la s}}^{4R} V(\rho) Z(\rho)\rho^{5/2}e^{-\frac{(\la\rho-s)^2}{4t}}\left(1+O\left(\frac{2t}{\lambda \rho s}\right)\right) d\rho\notag\\
&~\quad + \int_0^{4R}V(\rho)Z(\rho)\rho^5 e^{-\frac{\la^2\rho^2}{4t}}d\rho\int_0^{\frac{2t}{\la\rho}} O\left(\frac{\lambda^2\rho^2s^7}{254t^3}\right)\Theta_0(s) e^{-\frac{s^2}{4t}}ds\notag\\
&~ :=I_1+I_2+I_3+I_4 + I_5.\notag
\end{align}
Here $Z(\rho)=\frac1{(1+\rho^2)^2}$, $V(\rho)=\frac{24}{(1+\rho^2)^2}$.
For the term $I_1$, we have 
\begin{align*}
    I_1=&~\int_0^{\sqrt t}\frac{s^5}{64t^3} e^{-\frac{s^2}{4t}}\Theta_0(s)ds\int_0^{4R}V(\rho) Z(\rho)\rho^5 e^{-\frac{\la^2\rho^2}{4t}} d\rho\\
    =&~\left[C_1+O(R^{-2})+O(\la^2R^2/t)\right]\int_0^{\sqrt t}\frac{s^5}{64t^3} e^{-\frac{s^2}{4t}}\Theta_0(s)ds,
\end{align*}
where $C_1=\int_0^\infty V(\rho) Z(\rho)\rho^5 d\rho=4$. Furthermore, we have 
\begin{align*}
    I_2=&~\left[C_1+O(R^{-2})+O(\la^2R^2/t)\right]\int_{\sqrt t}^{\frac{t}{2\la R}}\frac{s^5}{64t^3} e^{-\frac{s^2}{4t}}\Theta_0(s)ds, \\
    I_3=&~\int_{\frac{t}{2\la R}}^\infty \frac{s^5}{64t^3} e^{-\frac{s^2}{4t}}\Theta_0(s)ds \int_0^{\frac{2t}{\la s}} V(\rho) Z(\rho)\rho^5 e^{-\frac{\la^2\rho^2}{4t}} d\rho\\
    =&~\left(\int_{\frac{t}{2\la R}}^{\frac{2t}{\la}}+\int_{\frac{2t}{\la}}^\infty\right) \frac{s^5}{64t^3} e^{-\frac{s^2}{4t}}\Theta_0(s)ds \int_0^{\frac{2t}{\la s}} V(\rho) Z(\rho)\rho^5 e^{-\frac{\la^2\rho^2}{4t}} d\rho\\
    =&~  C_1 \int_{\frac{t}{2\la R}}^{\frac{2t}{\la}}\frac{s^5}{64t^3} e^{-\frac{s^2}{4t}}\Theta_0(s) ds+\int_{\frac{2t}{\la}}^\infty \frac{s^5}{64t^3} e^{-\frac{s^2}{4t}}\Theta_0(s) \frac{256t^6}{\la^6 s^6}\frac{1}{(1+\frac{4t^2}{\la^2 s^2})^3} ds,\\
    I_4=&~\la^{-5/2}\frac{\sqrt{2}}{\sqrt{t}}\int_{\frac{t}{2\la R}}^\infty s^{5/2}\Theta_0(s) ds \int_{\frac{2t}{\la s}}^{4R} V(\rho) Z(\rho)\rho^{5/2}e^{-\frac{(\la\rho-s)^2}{4t}} d\rho,\\
    I_5 = &~\int_0^{4R}V(\rho)Z(\rho)\rho^5 e^{-\frac{\la^2\rho^2}{4t}}d\rho\int_0^{\frac{2t}{\la\rho}} O\left(\frac{\lambda^2\rho^2s^7}{254t^5}\right)\Theta_0(s) e^{-\frac{s^2}{4t}}ds.
\end{align*}
Now from the assumption that $\Theta_0(r)$ decays like
$$
|\Theta_0(r)|\sim r^{-2}\log^{-a} r ~\mbox{ as }~r\to+\infty,\quad 0<a<1.
$$
One estimates as follows,
\begin{align*}
    |I_3|\lesssim&~t^{-1}(\log t)^{-a}e^{-\frac{t}{\la^2 R^2}}(\frac{t}{\la^2 R^2}+1)+\la^{\epsilon}t^{-1-\epsilon/2}(\log t)^{-a},\\
    |I_4|\lesssim&~ \la^{-5/2}\frac1{\sqrt{2t}}\int_{\frac{t}{2\la R}}^\infty s^{5/2}|\Theta_0(s)| e^{-\frac{s^2}{4t}} ds\lesssim t^{-1-\epsilon/2}(\log t)^{-a} \la^{\epsilon}R^{5/2+\epsilon},\\
    |I_5|&~\lesssim \int_0^{4R}V(\rho)Z(\rho)\rho^5 e^{-\frac{\la^2\rho^2}{4t}}d\rho\int_0^{\frac{2t}{\la\rho}} \frac{s^7}{t^5}\Theta_0(s) e^{-\frac{s^2}{4t}}ds\lesssim \frac{\lambda^2R^2}{t^2}
\end{align*}
for $\epsilon>0$ small. Above estimates together with \eqref{orthog-214214} therefore imply
\begin{align}
&\log\la(t)\notag\\
&~= -\frac{3}{32}\int_{t_0}^t \frac1{\eta^3}d\eta\int_0^{\sqrt{\eta}}s^5 e^{-\frac{s^2}{4\eta}}\Theta_0(s)ds - \frac{3}{2} \int_{t_0}^t \sum_{i=2}^5I_i(\eta)d\eta  +\log\la(t_0)\notag\\
&~=-\frac{3}{32}\int_0^{\sqrt{t_0}} s^5\Theta_0(s) ds\int_{t_0}^t \frac{1}{\eta^3} e^{-\frac{s^2}{4\eta}} d\eta-\frac{3}{32}\int_{\sqrt{t_0}}^{\sqrt t} s^5\Theta_0(s) ds\int_{s^2}^t \frac{1}{\eta^3} e^{-\frac{s^2}{4\eta}} d\eta\label{fubi-exchange}\\
&~\quad - \frac{3}{2} \int_{t_0}^t \sum_{i=2}^5I_i(\eta)d\eta +\log\la(t_0)\notag\\
&~= -\frac{3}{32}\int_0^{\sqrt{t_0}} s^5\Theta_0(s) ds\int_{t_0}^t \frac{1}{\eta^3} e^{-\frac{s^2}{4\eta}} d\eta  -\frac{3}{8}\int_{\sqrt{t_0}}^{\sqrt{t}}s\Theta_0(s)e^{-\frac{s^2}{4 t}} \left(\frac{s^2}{t}+4\right)ds\notag\\
&~\quad + \frac{15}{8}e^{-\frac{1}{4}}\int_{M\sqrt{t_0}}^{M\sqrt{t}}s\Theta_0(s)ds - \frac{3}{2} \int_{t_0}^t \sum_{i=2}^5I_i(\eta)d\eta + \log\la(t_0)\notag\\
&~= -\frac{3}{8}\int_0^{\sqrt{t_0}} s\Theta_0(s)\left(e^{-\frac{s^2}{4 t}} \left(\frac{s^2}{t}+4\right)-e^{-\frac{s^2}{4t_0}}\left(\frac{s^2}{t_0} + 4\right)\right)ds  -\frac{3}{8}\int_{\sqrt{t_0}}^{\sqrt{t}}s\Theta_0(s)e^{-\frac{s^2}{4 t}} \left(\frac{s^2}{t}+4\right)ds\notag\\
&~\quad + \frac{15}{8}e^{-\frac{1}{4}}\int_{\sqrt{t_0}}^{\sqrt{t}}s\Theta_0(s)ds - \frac{3}{2} \int_{t_0}^t \sum_{i=2}^5I_i(\eta)d\eta + \log\la(t_0)\notag\\
&~= -\frac{3}{8}\int_{0}^{\sqrt{t}}s\Theta_0(s)e^{-\frac{s^2}{4 t}} \left(\frac{s^2}{t}+4\right)ds + \frac{3}{8}\int_0^{\sqrt{t_0}} s\Theta_0(s)e^{-\frac{s^2}{4t_0}}\left(\frac{s^2}{t_0} + 4\right)ds\notag\\
&~\quad + \frac{15}{8}e^{-\frac{1}{4}}\int_{\sqrt{t_0}}^{\sqrt{t}}s\Theta_0(s)ds - \frac{3}{2} \int_{t_0}^t \sum_{i=2}^5I_i(\eta)d\eta + \log\la(t_0).\notag
\end{align}
Here we have used the integral identities,
$$
\int_{s^2}^t \frac{1}{\eta^3} e^{-\frac{s^2}{4\eta}} d\eta = \frac{4 \left(e^{-\frac{s^2}{4 t}} \left(\frac{s^2}{t}+4\right)-5e^{-\frac{1}{4}}\right)}{s^4}
$$
and
$$
\int_{t_0}^t \frac{1}{\eta^3} e^{-\frac{s^2}{4\eta}} d\eta = \frac{4 \left(e^{-\frac{s^2}{4 t}} \left(\frac{s^2}{t}+4\right)-\frac{e^{-\frac{s^2}{4t_0}}\left(4 t_0+s^2\right)}{t_0}\right)}{s^4}.
$$
By decomposing the term $\int_0^{\sqrt{t}} s\Theta_0(s)e^{-\frac{s^2}{4 t}}ds$ as 
\begin{align*}
\int_0^{\sqrt{t}} s\Theta_0(s)e^{-\frac{s^2}{4 t}}ds & = \int_0^{\sqrt{t}} s\Theta_0(s)ds + \int_0^{\sqrt{t}} s\Theta_0(s)\left(e^{-\frac{s^2}{4 t}} - 1\right)ds\\
& = \int_{\sqrt{t_0}}^{\sqrt{t}} s\Theta_0(s)ds + \int_0^{\sqrt{t_0}} s\Theta_0(s)ds + \int_0^{\sqrt{t}} s\Theta_0(s)\left(e^{-\frac{s^2}{4 t}} - 1\right)ds,
\end{align*}
we have
\begin{align*}
&\log\la(t)\\
&~= -\frac{3}{2}\int_{\sqrt{t_0}}^{\sqrt{t}} s\Theta_0(s)ds -\frac{3}{2}\int_0^{\sqrt{t_0}} s\Theta_0(s)ds -\frac{3}{2}\int_0^{\sqrt{t}} s\Theta_0(s)\left(e^{-\frac{s^2}{4 t}} - 1\right)ds\\
&~\quad -\frac{3}{8}\int_{0}^{\sqrt{t}}s\Theta_0(s)e^{-\frac{s^2}{4 t}} \frac{s^2}{t}ds + \frac{3}{8}\int_0^{\sqrt{t_0}} s\Theta_0(s)e^{-\frac{s^2}{4t_0}}\left(\frac{s^2}{t_0} + 4\right)ds\\
&~\quad + \frac{15}{8}e^{-\frac{1}{4}}\int_{\sqrt{t_0}}^{\sqrt{t}}s\Theta_0(s)ds - \frac{3}{2} \int_{t_0}^t \sum_{i=2}^5I_i(\eta)d\eta + \log\la(t_0)\\
&~= -\frac{3}{2}\int_{\sqrt{t_0}}^{\sqrt{t}} s\Theta_0(s)ds + \frac{15}{8}e^{-\frac{1}{4}}\int_{\sqrt{t_0}}^{\sqrt{t}}s\Theta_0(s)ds- \frac{3}{2} \int_{t_0}^t I_2(\eta)d\eta\\
&~\quad -\frac{3}{2}\int_0^{\sqrt{t_0}} s\Theta_0(s)ds -\frac{3}{2}\int_0^{\sqrt{t}} s\Theta_0(s)\left(e^{-\frac{s^2}{4 t}} - 1\right)ds\\
&~\quad -\frac{3}{8}\int_{0}^{\sqrt{t}}s\Theta_0(s)e^{-\frac{s^2}{4 t}} \frac{s^2}{t}ds + \frac{3}{8}\int_0^{\sqrt{t_0}} s\Theta_0(s)e^{-\frac{s^2}{4t_0}}\left(\frac{s^2}{t_0} + 4\right)ds\\
&~\quad  - \frac{3}{2} \int_{t_0}^t \sum_{i=3}^5I_i(\eta)d\eta + \log\la(t_0).
\end{align*}
Observe that we have 
$$
\left|\int_0^{\sqrt{t}} s\Theta_0(s)\left(e^{-\frac{s^2}{4 t}} - 1\right)ds\right|\lesssim C,\quad \left|\int_0^{\sqrt{t}} s\Theta_0(s)e^{-\frac{s^2}{4 t}} \frac{s^2}{t}ds\right|\lesssim C,\quad \left|\int_{t_0}^t \sum_{i=3}^5I_i(\eta)d\eta\right|\lesssim C,
$$
from which we know that
\begin{align*}
\log\la(t)= -\frac{3}{2}\int_{\sqrt{t_0}}^{\sqrt{t}} s\Theta_0(s)ds + \frac{15}{8}e^{-\frac{1}{4}}\int_{\sqrt{t_0}}^{\sqrt{t}}s\Theta_0(s)ds- \frac{3}{2} \int_{t_0}^t I_2(\eta)d\eta + \tilde C_1(t_0, t).
\end{align*}
Here 
\begin{align*}
\tilde C_1(t_0, t):& = -\frac{3}{2}\int_0^{\sqrt{t_0}} s\Theta_0(s)ds -\frac{3}{2}\int_0^{\sqrt{t}} s\Theta_0(s)\left(e^{-\frac{s^2}{4 t}} - 1\right)ds -\frac{3}{8}\int_{0}^{\sqrt{t}}s\Theta_0(s)e^{-\frac{s^2}{4 t}} \frac{s^2}{t}ds \\
&~\quad + \frac{3}{8}\int_0^{\sqrt{t_0}} s\Theta_0(s)e^{-\frac{s^2}{4t_0}}\left(\frac{s^2}{t_0} + 4\right)ds - \frac{3}{2} \int_{t_0}^t \sum_{i=3}^5I_i(\eta)d\eta
\end{align*}
is a smooth bounded function of $t\in (t_0, \infty)$.

Now we deal with the term $\int_{t_0}^t I_2(\eta)d\eta$. Write
$$
I_2=\left[C_1+O(R^{-2})+O(\la^2R^2/t)\right]\left(\int_{\sqrt t}^{t^{\frac{3}{4}}}\frac{s^5}{64t^3} e^{-\frac{s^2}{4t}}\Theta_0(s)ds + \int_{t^{\frac{3}{4}}}^{\frac{t}{2\la R}}\frac{s^5}{64t^3} e^{-\frac{s^2}{4t}}\Theta_0(s)ds\right) := I_{21} + I_{22}.
$$
Firstly, using the Fubini theorem as in (\ref{fubi-exchange}), it holds that 
\begin{align*}
&\int_{t_0}^t \int_{\sqrt \eta}^{\eta^{\frac{3}{4}}}\frac{s^5}{64\eta^3} e^{-\frac{s^2}{4\eta}}\Theta_0(s)dsd\eta\\
&= \frac{1}{16}\int_{\sqrt{t_0}}^{t_0^{\frac{3}{4}}}s\Theta_0(s)\left(5e^{-\frac{1}{4}}-e^{-\frac{s^2}{4t_0}}\left(4+\frac{s^2}{t_0}\right)\right)ds + \frac{1}{16}\int_{t_0^{\frac{3}{4}}}^{\sqrt{t}}s\Theta_0(s)\left(5e^{-\frac{1}{4}}-e^{-\frac{s^{\frac{2}{3}}}{4}}\left(4+s^{\frac{2}{3}}\right)\right)ds\\
&\quad + \frac{1}{16}\int_{\sqrt{t}}^{t^{\frac{3}{4}}}s\Theta_0(s)\left(e^{-\frac{s^2}{4 t}} \left(\frac{s^2}{t}+4\right)-e^{-\frac{s^{\frac{2}{3}}}{4}}\left(4+s^{\frac{2}{3}}\right)\right)ds\\
& = e^{-\frac{1}{4}}\frac{5}{16}\int_{\sqrt{t_0}}^{\sqrt{t}}s\Theta_0(s)ds + \frac{1}{16}\int_{\sqrt{t}}^{t^{\frac{3}{4}}}s\Theta_0(s)e^{-\frac{s^2}{4 t}} \left(\frac{s^2}{t}+4\right)ds\\
&\quad - \frac{1}{16}\int_{t_0^{\frac{3}{4}}}^{t^{\frac{3}{4}}}s\Theta_0(s)e^{-\frac{s^{\frac{2}{3}}}{4}}\left(4+s^{\frac{2}{3}}\right)ds -\frac{1}{16}\int_{\sqrt{t_0}}^{t_0^{\frac{3}{4}}}s\Theta_0(s)e^{-\frac{s^2}{4t_0}}\left(4+\frac{s^2}{t_0}\right)ds.
\end{align*}
Note that the terms 
\begin{align*}
&\int_{t_0^{\frac{3}{4}}}^{t^{\frac{3}{4}}}s\Theta_0(s)e^{-\frac{s^{\frac{2}{3}}}{4}}\left(4+s^{\frac{2}{3}}\right)ds +\int_{\sqrt{t_0}}^{t_0^{\frac{3}{4}}}s\Theta_0(s)e^{-\frac{s^2}{4t_0}}\left(4+\frac{s^2}{t_0}\right)ds
\end{align*}
are bounded functions of $t$ and $t_0$, and 
$$
\left|\int_{\sqrt{t}}^{t^{\frac{3}{4}}}s\Theta_0(s)e^{-\frac{s^2}{4 t}} \left(\frac{s^2}{t}+4\right)ds\right|\lesssim C.
$$
Secondly, we have 
\begin{align*}
&\left|I_{22}\right| = \left[C_1+O(R^{-2})+O(\la^2R^2/t)\right]\left|\int_{t^{\frac{3}{4}}}^{\frac{t}{2\la R}}\frac{s^5}{64t^3} e^{-\frac{s^2}{4t}}\Theta_0(s)ds\right|\\
& \lesssim \int_{t^{\frac{3}{4}}}^{\frac{t}{2\la R}}\frac{s^3}{t^3} e^{-\frac{s^2}{4t}}\frac{1}{(\log s)^a}ds \lesssim \frac{1}{t}\int_{t^{\frac{1}{4}}}^{\frac{\sqrt{t}}{2\lambda R}}s^3 e^{-\frac{s^2}{4}}\frac{1}{(\log s + \frac{1}{2}\log t)^a}ds\\ &\lesssim \frac{1}{t}\frac{1}{(\log t)^a}\int_{t^{\frac{1}{4}}}^{\frac{\sqrt{t}}{2\lambda R}}s^3 e^{-\frac{s^2}{4}}ds\lesssim \frac{1}{t^2(\log t)^{a}},
\end{align*}
and
$$
\left|\int_{t_0}^t I_{22}(\eta)d\eta\right|\lesssim \frac{1}{t_0}.
$$
Therefore, one has 
\begin{align*}
\int_{t_0}^t I_2(\eta)d\eta = \frac{5}{4}e^{-\frac{1}{4}}\int_{\sqrt{t_0}}^{\sqrt{t}}s\Theta_0(s)ds + \tilde C_2(t_0, t)
\end{align*}
with $\tilde C_2(t_0, t)$ being a smooth bounded function of $t\in (t_0, \infty)$.

Combine all the estimates above, we obtain that 
\begin{align*}
\log\la(t)= \left(-\frac{3}{2}+O\left(\frac{1}{\log t}\right)\right)\int_{\sqrt{t_0}}^{\sqrt{t}} s\Theta_0(s)ds + C(t_0, t).
\end{align*}
Here $C(t_0, t): = \tilde C_1(t_0, t)-\frac{3}{2}\tilde C_2(t_0, t)$ is a smooth bounded function of $t\in (t_0, \infty)$. This concludes the proof.
\end{proof}
Now we claim that 
\begin{lemma}\label{choice-of-Theta_0}
There exists sign-changing initial data $\Theta_0$ with $$\Theta_0(r)\sim r^{-2}\log^{-a} r ~\mbox{ as }~r\to+\infty,\quad a\in (0, 1)$$
such that the solution of (\ref{orthog-214214}) satisfies
$$ 
\liminf_{t\to +\infty}\lambda(t) = 0\quad \text{and}\quad \limsup_{t\to +\infty}\lambda(t) = \infty.
$$
\end{lemma}
\begin{proof}
Let us set
\begin{align*}
\Theta_0(r) & = (1-a)\frac{\cos(\log\log(2+r^2))}{(2+r^2)\left[\log(2+r^2)\right]^a} - \frac{\sin(\log\log(2+r^2))}{(2+r^2)\left[\log(2+r^2)\right]^a}\\ &= \frac{(1-a)\cos(\log\log(2+r^2))-\sin(\log\log(2+r^2))}{(2+r^2)\left[\log(2+r^2)\right]^a}.
\end{align*}
with $a \in (0, 1)$. This function is sign-changing and satisfies $$\Theta_0(r)\sim r^{-2}\log^{-a} r ~\mbox{ as }~r\to+\infty.$$
From (\ref{GNK-law}), we know that the main term $\lambda_0$ of $\lambda$ can be computed as 
\begin{align*}
&\log\la_0(t) \\
& =-\frac{3}{2}\int_{\sqrt{t_0}}^{\sqrt{t}}s\Theta_0(s)ds \\
&=-\frac{3}{4}\int_{t_0}^{t}\left((1-a)\frac{\cos(\log\log(2+s))}{(2+s)\left[\log(2+s)\right]^a} - \frac{\sin(\log\log(2+s))}{(2+s)\left[\log(2+s)\right]^{a}}\right)d s\\
&= -\frac{3}{4}\int_{t_0}^{t}d \left[\left(\log(2+s)\right)^{1-a}\cos(\log\log(2+s))\right]\\
&= -\frac{3}{4}\left(\log(2+t)\right)^{1-a}\cos(\log\log(2+t)) + \frac{3}{4}\left(\log(2+t_0)\right)^{1-a}\cos(\log\log(2+t_0)), 
\end{align*}
therefore, we get 
\begin{align}\label{main-term-of-lambda}
&\la_0(t)\notag\\ 
& = \exp\left(-\frac{3}{4}\left(\log(2+t)\right)^{1-a}\cos(\log\log(2+t)) + \frac{3}{4}\left(\log(2+t_0)\right)^{1-a}\cos(\log\log(2+t_0))\right)\\ 
& = C_{t_0}\exp\left(-\frac{3}{4}\left(\log(2+t)\right)^{1-a}\cos(\log\log(2+t))\right),\notag
\end{align}
with $C_{t_0} = \exp\left(\frac{3}{4}\left(\log(2+t_0)\right)^{1-a}\cos(\log\log(2+t_0))\right)$. This shows that $\lambda_0(t) > 0$,
$$ 
\liminf_{t\to +\infty}\lambda_0(t) = 0\quad \text{and}\quad \limsup_{t\to +\infty}\lambda_0(t) = \infty
$$
for $a \in (0, 1)$. From (\ref{GNK-law}), we also know that 
\begin{align*}
\left|\frac{\dot\lambda_0(t)}{\lambda_0(t)}\right| = \left|-\frac{3}{4}\Theta_0(\sqrt{t})\right|\sim t^{-1}(\log t)^{-a}
\end{align*}
as $t\to +\infty$.
\end{proof}

\begin{remark}
To solve the inner problem \eqref{eqn-inner}, we use a new time variable as follows,
\begin{equation}\label{time-variable-transformation}
\tau = \tau(t) : = \int_{t_0}^t\lambda^{-2}(s)ds + C t_0\lambda^{-2}(t_0), \quad \tau_0: = \tau(t_0).
\end{equation}
For $\lambda$ with main term defined in (\ref{main-term-of-lambda}), we have $\tau(t)\to +\infty$ as $t\to+\infty$. Indeed, we have 
\begin{align*}
\tau &=\tau(t)\\ & =C_{t_0}^{-2}\int_{t_0}^t \exp\left(\frac{3}{2}\left(\log(2+s)\right)^{1-a}\cos(\log\log(2+s))\right)ds \\
& = C_{t_0}^{-2} \int_{\log\log(2+t_0)}^
{\log\log(2+t)} \exp(\exp(s))\exp(s)\exp\left(\frac{3}{2}\exp((1-a)s)\cos(s)\right)ds \\
&\geq C_{t_0}^{-2} \int_{\cos s < 1/2} \exp(\exp(s))\exp(s)\exp\left(\frac{3}{4}\exp((1-a)s)\right)ds\\
& \notag=  C_{t_0}^{-2} \sum_{k=[\frac{\log\log(2+t_0)}{\pi}]}^{[\frac{\log\log(2+t)}{\pi}]}\int_{-\frac{\pi}{3}+k\pi}^{\frac{\pi}{3}+k\pi}\exp(\exp(s))\exp(s)\exp\left(\frac{3}{4}\exp((1-a)s)\right)ds\\
&\notag \geq C_{t_0}^{-2}\sum_{k=[\frac{\log\log(2+t_0)}{\pi}]}^{[\frac{\log\log(2+t)}{\pi}]}\exp(\exp(-\frac{\pi}{3}+k\pi))\exp(-\frac{\pi}{3}+k\pi)\exp\left(\frac{3}{4}\exp\left((1-a)\left(-\frac{\pi}{3}+k\pi\right)\right)\right)\frac{2\pi}{3}\\
&\to +\infty\quad{ as }\quad t\to +\infty.
\end{align*}
\end{remark}
Now we define 
\begin{equation*}
    \Theta_*(\tau):=\tilde\Theta_*(t(\tau)):=\frac{(\log t)^{-a}}{t}\gtrsim \left|
    \frac{\dot\lambda_0}{\lambda_0}\right|.
\end{equation*}
The RHS of the inner problem \eqref{eqn-inner}, expressed in the $(y,\tau)$ variables, behaves like
    $$
    |H(\rho,\tau)|\lesssim |\Theta_*|\langle \rho\rangle^{-2-a},\quad 0<a<1,
    $$
    so Proposition \ref{Prop-inner} gives
    $$
    \langle\rho\rangle|\phi_\rho(\rho,\tau)|+|\phi(\rho,\tau)|\lesssim \lambda^2|\Theta_*|R^{6-a}\langle \rho\rangle^{-6}.
    $$
We define 
\begin{equation*}
    \begin{aligned}\|\phi\|_{\sharp}:=&~\sup_{\rho,\tau} \lambda^{-2}(\tau) |\Theta_*(\tau)|^{-1} R^{a-6}(\tau)\langle \rho\rangle^{6}\Big(\langle\rho\rangle|\phi_\rho(\rho,\tau)|+|\phi(\rho,\tau)|\Big),\\
    \|\psi\|_\flat:=&~\sup_{r,t}\Big[|\Theta_*(\tau(t))|R^{-a_1}(\tau(t))(\1_{r\leq \sqrt t}+tr^{-2}\1_{r\geq\sqrt t})\Big]^{-1}|\psi(r,t)|
    \end{aligned}
\end{equation*}
for $0<a_1<a$. Let us make assumptions on the parameters:
$$
\quad R = R(t) = (\log t)^{\frac{1}{100}}, \quad 
\lambda(t) = \lambda_0(t) + \lambda_1(t),
$$
with $\lambda_1(t)  = \frac{1}{(\log t)^{\epsilon}}\lambda_0(t)$. Then we have $\la R\ll\sqrt{t}$. The orthogonality condition \eqref{inner-orthogonal-cond} will be fully solved in Section \ref{oscillating-scale}. 

\medskip

\section{The outer problem}\label{outer-problem}
In this section, we consider the outer problem.
\begin{equation}\label{eqn-outer1}
    \begin{aligned}
\psi_t=&~\psi_{rr}+\frac5r\psi_r+ \mathcal G[\psi, \phi, \lambda_1] ~\mbox{ in }~\R_+\times (t_0,\infty),
    \end{aligned}
\end{equation}
where
\begin{equation}\label{eqn-outer1111}
    \begin{aligned}
\mathcal G[\psi, \phi, \lambda_1]=&~2\la^{-2}\pp_r\eta_R\phi_r+\la^{-2}\phi\left(\pp_{rr}\eta_R+\frac5r\pp_r\eta_R\right)+2\la^{-3}\dot\la \eta_R \phi + \la^{-3}\dot\la \eta_R \rho\phi_\rho+(1-\eta_R)\frac{4\dot\la \la^{-3}}{(1+\rho^2)^2}\\
    &~+(1-\eta_R)\mathcal N[\la,\phi,\psi] + \la^{-2}(1-\eta_R)(12U-6\rho^2 U^2)(\psi+\Theta)~\mbox{ in }~\R_+\times (t_0,\infty)
    \end{aligned}
\end{equation}
and 
\begin{equation}\label{def-N1}
    \begin{aligned}
        \mathcal N[\la,\phi,\psi]:=&~6(U_\la+\Theta+\eta_R\la^{-2}\phi+\psi)^2-2r^2(U_\la+\Theta+\eta_R\la^{-2}\phi+\psi)^3-6U_\la^2+2r^2 U_\la^3\\
        &~-(12U_\la-6r^2 U_\la^2)(\Theta+\eta_R\la^{-2}\phi+\psi).
    \end{aligned}
\end{equation}

\begin{proposition}
For $t_0$ sufficiently large, there exists a unique solution $\psi = \psi[\phi, \lambda_1]$ for the outer problem (\ref{eqn-outer1}) satisfying the following estimates
\begin{align}\label{psi-exists}
|\psi|\lesssim |\Theta_*(\tau(t))|R^{-a_1}(\tau(t))\left(\1_{r\leq \sqrt t}+tr^{-2}\1_{r\geq\sqrt t}\right),
\end{align}
\begin{align}\label{psi-gradient}
\|\psi_r(\cdot, t)\|_{L^\infty(\mathbb{R}_+)}\lesssim |\Theta_*(\tau(t))|R^{-a_1}(\tau(t)).
\end{align}
\end{proposition}
\begin{proof}
In \eqref{eqn-outer1}, one can estimate the right hand side terms of (\ref{eqn-outer1}) as follows,
\begin{align*}
g_1 &~:=\left|2\la^{-2}\pp_r\eta_R\phi_r+\la^{-2}\phi\left(\pp_{rr}\eta_R+\frac5r\pp_r\eta_R\right)+2\la^{-3}\dot\la \eta_R \phi+\la^{-3}\dot\la \eta_R \rho\phi_\rho\right|\\
&~\lesssim\left(\la^{-2}|\Theta_*|R^{4-a}\langle \rho\rangle^{-6}\1_{\{\rho\sim R\}}+|\Theta_*|^2 R^{6-a}\langle \rho\rangle^{-6}\1_{\{\rho\lesssim R\}}\right)\|\phi\|_{\sharp},\\
&~g_2:=\left|(1-\eta_R)\frac{4\dot\la \la^{-3}}{(1+\rho^2)^2}\right|\lesssim \la^{-2}|\Theta_*|\langle\rho\rangle^{-4}\1_{\{\rho\gtrsim R\}},\\
g_3&~:=|(1-\eta_R)\mathcal N[\la,\phi,\psi]|\\
&~=(1-\eta_R)\Big|6(\Theta+\la^{-2}\eta_R\phi+\psi)^2-2r^2[3U_\la(\Theta+\la^{-2}\eta_R\phi+\psi)^2+(\Theta+\la^{-2}\eta_R\phi+\psi)^3]\Big|\\
&~\lesssim |\Theta_*|^2 R^{12-2a}\langle \rho\rangle^{-12}\|\phi\|_{\sharp}^2\1_{\{\rho\sim R\}}+|\Theta|^2\1_{\{\rho\gtrsim R\}}\\
&\quad +\Big[|\Theta_*|^2 R^{-2a_1}(\1_{r\leq \sqrt t}+tr^{-2}\1_{r\geq\sqrt t})^2\Big]\|\psi\|^2_\flat\1_{\{\rho\gtrsim R\}}\\
&\quad + \la^{2}|\Theta_*|^3 R^{18-3a}\langle \rho\rangle^{-16}\|\phi\|_{\sharp}^3\1_{\{\rho\sim R\}}+\la^{2}\rho^2|\Theta|^3\1_{\{\rho\gtrsim R\}}\\
&\quad +\la^{2}\rho^2\Big[|\Theta_*|^3 R^{-3a_1}(\1_{r\leq \sqrt t}+tr^{-2}\1_{r\geq\sqrt t})^3\Big]\|\psi\|^3_\flat\1_{\{\rho\gtrsim R\}},\\
g_4 &~:=|\la^{-2}(1-\eta_R)(12 U-6\rho^2 U^2)(\psi+\Theta)| \\
&~\lesssim \la^{-2}|\Theta_*|\langle\rho\rangle^{-4}R^{-a_1}(\1_{r\leq \sqrt t}+tr^{-2}\1_{r\geq\sqrt t})\1_{\{\rho\gtrsim R\}}\|\psi\|_\flat + \la^{-2}|\Theta|\langle\rho\rangle^{-4}\1_{\{\rho\gtrsim R\}}.
\end{align*}
In the following, we use the case $n = 6$, $2< b < 6$ and $a\in [0, 1)$ of Lemma \ref{linear-outer},
\begin{equation*}
		\begin{aligned}
			&\TT_6^{{\rm out}}\left[ v(t) |x|^{-b}(\log |x|)^{-a} \1_{\{ l_1(t) \le |x| \le l_2(t) \}}\right]\\
			&\lesssim 
			t^{-3}
			e^
			{-\frac{|x|^2}{16 t } } 
			\int_{\frac{t_0}{2}}^{\frac{t}{2} }   v(s)
				l_2^{6-b}(s)(\log l_2(s))^{-a}
				d s 
			\\
			&~\quad +
			\sup\limits_{ t_1 \in [t/2,t] } v(t_1)
            \begin{cases}
			l_1^{2-b}(t)(\log l_1(t))^{-a}
				&
				\mbox{ \ for \ }  r \le l_1(t)
				\\
					r^{2-b}(\log r)^{-a}
				&
				\mbox{ \ for \ }
				l_1(t) < r \le l_2(t)
				\\
				r^{-4} 
				e^{-\frac{r^2}{16 t}}
					l_2^{6-b}(t) (\log l_2(t))^{-a}
				&
				\mbox{ \ for \ }
				r > l_2(t),
			\end{cases} 
		\end{aligned}
	\end{equation*}

    \begin{equation*}
		\begin{aligned}
			&
			\TT_{6}^{{\rm out}}\left[v(t) |x|^{-b}(\log |x|)^{-a} \1_{\{ |x|\ge t^{\frac 12}\}}\right]\\
			&\lesssim  
			\begin{cases}
				t^{-3}
				\int_{t_{0}/2}^{t/2}
				v(s)
					t^{\frac{6-b}{2}}(\log t)^{-a}
				d s  
				  + t^{1-\frac{b}{2}}(\log t)^{-a} \sup\limits_{ t_1 \in [t/2,t] } v(t_1) 
				&
				\mbox{ \ if \ } |x|\le t^{\frac 12}
				\\
				|x|^{-b} (\log |x|)^{-a}
				\left( t \sup\limits_{ t_1 \in [t/2,t] } v(t_1) +	\int_{t_{0}/2}^{t/2}
				v(s) d s \right) 
				&
				\mbox{ \ if \ } |x| > t^{\frac 12}.
			\end{cases} 
		\end{aligned}
	\end{equation*}
Furthermore, when $n = 6$, $b = 6$ and $a\in [0, 1)$, one has
\begin{equation*}
		\begin{aligned}
			&
			\TT_{n}^{out}\left[v(t) |x|^{-6} (\log |x|)^{-a}\1_{\{ r\ge t^{\frac 12}\}}\right]\lesssim
			\\
			&
			\begin{cases}
				t^{-3}
				\int_{t_{0}/2}^{t/2}
				v(s)(\log t)^{1-a}d s
				+ t^{-2}(\log t)^{-a} \sup\limits_{ t_1 \in [t/2,t] } v(t_1)
				\mbox{ \ if \ } |x|\le t^{\frac 12}
				\\ 
				  |x|^{-6}(\log t)^{-a}
				\left( t \sup\limits_{ t_1 \in [t/2,t] } v(t_1) +	\int_{t_{0}/2}^{t/2}
				v(s) d s \right)  +
				t^{-3} e^{-\frac{|x|^2}{16 t} }
				\int_{t_{0}/2}^{t/2}
				v(s)
				(\log(|x|))^{1-a}d s
				\mbox{ \ if \ } |x| > t^{\frac 12}.
			\end{cases}
		\end{aligned}
	\end{equation*}
Term by term, using Lemma \ref{linear-outer}, we estimate as follows.\\
\noindent{\bf Claim 1}:
\begin{align*}
&~\TT_6^{{\rm out}}[g_1]\\
&~\lesssim\|\phi\|_\sharp\TT_6^{{\rm out}}\Big[\left(\la^{-2}|\Theta_*|R^{4-a}\langle \rho\rangle^{-6}\1_{\{\rho\sim R\}}+|\Theta_*|^2 R^{6-a}\langle \rho\rangle^{-6}\1_{\{\rho\lesssim R\}}\right)\Big]\\
&~\lesssim (\log t)^{-\epsilon}|\Theta_*(\tau(t))|R^{-a_1}(\tau(t))(\1_{r\leq \sqrt t}+tr^{-2}\1_{r\geq\sqrt t}).
\end{align*}
{\it Proof}: Indeed, we have
\begin{align*}
&\TT_6^{{\rm out}}\Big[\la^{-2}|\Theta_*|R^{-2-a}\1_{\{\rho\sim R\}}\Big]\lesssim \TT_6^{{\rm out}}\Big[\la^{-2}|\Theta_*|R^{-2-a}\1_{\{\frac{1}{2}\lambda R \leq r \leq 2\lambda R\}}\Big]\\
&\lesssim t^{\epsilon-3}e^{-\frac{r^2}{16t}} + |\Theta_*| R^{-a}\1_{\{r\leq 2\lambda R\}} + |\Theta_*| R^{-a}\lambda^4R^4|x|^{-4}e^{-\frac{r^2}{16t}}\1_{\{r\geq 2\lambda R\}} \\ 
& \lesssim (\log t)^{-\epsilon}|\Theta_*(\tau(t))|R^{-a_1}(\tau(t))(\1_{r\leq \sqrt t}+tr^{-2}\1_{r\geq\sqrt t}),
\end{align*}

\begin{align*}
&\TT_6^{{\rm out}}\Big[|\Theta_*|^2 R^{6-a}\langle \rho\rangle^{-6}\1_{\{\rho\lesssim R\}}\Big]\lesssim\TT_6^{{\rm out}}\Big[|\Theta_*|^2 R^{6-a}\1_{\{r\leq \la R\}}\Big]\\
&\lesssim t^{-3}e^{-\frac{r^2}{16t}} + |\Theta_*|^2 R^{6-a}\lambda^2R^2\1_{\{r\leq 2\lambda R\}} + |\Theta_*|^2 R^{6-a}\lambda^6R^6r^{-4}e^{-\frac{r^2}{16t}}\1_{\{r\geq 2\lambda R\}}\\ 
& \lesssim (\log t)^{-\epsilon}|\Theta_*(\tau(t))|R^{-a_1}(\tau(t))(\1_{r\leq \sqrt t}+tr^{-2}\1_{r\geq\sqrt t}).
\end{align*}

\noindent{\bf Claim 2}:
\begin{align*}
\TT_6^{{\rm out}}[g_2]\lesssim (\log t)^{-\epsilon}|\Theta_*(\tau(t))|R^{-a_1}(\tau(t))(\1_{r\leq \sqrt t}+tr^{-2}\1_{r\geq\sqrt t}).
\end{align*}
{\it Proof}: We have
\begin{align*}
&\TT_6^{{\rm out}}\Big[\la^{-2}|\Theta_*|\langle\rho\rangle^{-4}\1_{\{\rho\gtrsim R\}}\Big]\\ 
& \lesssim \TT_6^{{\rm out}}\Big[\la^{2}|\Theta_*|r^{-4}\1_{\{\lambda R \leq r \leq \sqrt{t}\}}\Big] + \TT_6^{{\rm out}}\Big[\la^{2}|\Theta_*|r^{-4}\1_{\{r\gtrsim \sqrt{t}\}}\Big]\\
& \lesssim \frac{t^{-2+\epsilon}}{(\log t)^a}e^{-\frac{|x|^2}{16t}} + \frac{|\Theta_*|}{R^2} \1_{r\leq \sqrt t} + \frac{\la^{2}|\Theta_*|}{t} tr^{-2}\1_{r\geq\sqrt t} + \frac{\la^{2}|\Theta_*|}{t}\1_{r\leq \sqrt t} + \frac{\la^{2}|\Theta_*|}{t} tr^{-2}\1_{r\geq\sqrt t}\\
& \lesssim (\log t)^{-\epsilon}|\Theta_*(\tau(t))|R^{-a_1}(\tau(t))(\1_{r\leq \sqrt t}+tr^{-2}\1_{r\geq\sqrt t}). 
\end{align*}

\noindent{\bf Claim 3}:
\begin{align*}
\TT_6^{{\rm out}}[g_3]&~ \lesssim \TT_6^{{\rm out}}[|\Theta_*|^2 R^{12-2a}\langle \rho\rangle^{-12}\|\phi\|_{\sharp}^2\1_{\{\rho\sim R\}}+|\Theta|^2\1_{\{\rho\gtrsim R\}}\\
&\quad +\Big[|\Theta_*|^2 R^{-2a_1}(\1_{r\leq \sqrt t}+tr^{-2}\1_{r\geq\sqrt t})^2\Big]\|\psi\|^2_\flat\1_{\{\rho\gtrsim R\}}\\
&\quad + \la^{2}|\Theta_*|^3 R^{18-3a}\langle \rho\rangle^{-16}\|\phi\|_{\sharp}^3\1_{\{\rho\sim R\}}+\la^{2}\rho^2|\Theta|^3\1_{\{\rho\gtrsim R\}}\\
&\quad +\la^{2}\rho^2\Big[|\Theta_*|^3 R^{-3a_1}(\1_{r\leq \sqrt t}+tr^{-2}\1_{r\geq\sqrt t})^3\Big]\|\psi\|^3_\flat\1_{\{\rho\gtrsim R\}}]\\ &~\lesssim (\|\phi\|^2_\sharp+\|\psi\|^2_\flat)\TT_6^{{\rm out}}\Big[\Big[|\Theta_*|^2 R^{-2a_1}(\1_{r\leq \sqrt t}+tr^{-2}\1_{r\geq\sqrt t})^2\Big]\1_{\{\rho\gtrsim R\}}\Big]\\
&~\quad+\TT_6^{{\rm out}}\Big[|\Theta|^2\1_{\{\rho\gtrsim R\}}\Big] + \TT_6^{{\rm out}}[\la^{2}|\Theta_*|^3 R^{18-3a}\langle \rho\rangle^{-16}\|\phi\|_{\sharp}^3\1_{\{\rho\sim R\}}+\la^{2}\rho^2|\Theta|^3\1_{\{\rho\gtrsim R\}}\\
&~\quad +\la^{2}\rho^2\Big[|\Theta_*|^3 R^{-3a_1}(\1_{r\leq \sqrt t}+tr^{-2}\1_{r\geq\sqrt t})^3\Big]\|\psi\|^3_\flat\1_{\{\rho\gtrsim R\}}]\\
&~\lesssim (\|\phi\|^2_\sharp+\|\psi\|^2_\flat+\|\phi\|^3_\sharp+\|\psi\|^3_\flat)(\log t)^{-\epsilon}|\Theta_*(\tau(t))|R^{-a_1}(\tau(t))(\1_{r\leq \sqrt t}+tr^{-2}\1_{r\geq\sqrt t}) \\
&~\quad+ (\log t)^{-\epsilon}|\Theta_*(\tau(t))|R^{-a_1}(\tau(t))(\1_{r\leq \sqrt t}+tr^{-2}\1_{r\geq\sqrt t}).
\end{align*}
{\it Proof}: We have
\begin{align*}
&\TT_6^{{\rm out}}\Big[|\Theta|^2\1_{\{\rho\gtrsim R\}}\Big]\\
& \lesssim \TT_6^{{\rm out}}\Big[t^{-2}(\log t)^{-2a}\1_{\{\lambda R\leq r\leq\sqrt t\}}+ r^{-4}(\log r)^{-2a}\1_{\{r\geq\sqrt t\}}\Big]\\
& \lesssim t^{-1}(\log t)^{-2a}e^{-\frac{|x|^2}{16t}} + t^{-1}(\log t)^{-2a}\1_{\{r\leq\sqrt t\}} + t^{-1}(\log t)^{-2a}tr^{-2}\1_{\{r\geq\sqrt t\}}\\
& \quad + t^{-1}(\log t)^{-2a}\1_{\{r\leq\sqrt t\}} + r^{-2}(\log r)^{-2a}tr^{-2}\1_{\{r\geq\sqrt t\}}\\
& \lesssim (\log t)^{-\epsilon}|\Theta_*(\tau(t))|R^{-a_1}(\tau(t))(\1_{r\leq \sqrt t}+tr^{-2}\1_{r\geq\sqrt t}),  
\end{align*}

\begin{align*}
&\TT_6^{{\rm out}}\Big[|\Theta_*|^2 R^{-2a_1}(\1_{r\leq \sqrt t}+tr^{-2}\1_{r\geq\sqrt t})^2\1_{\{\rho\gtrsim R\}}\Big]\\
&\lesssim \TT_6^{{\rm out}}\Big[|\Theta_*|^2 R^{-2a_1}(\1_{\lambda R \leq r\leq \sqrt t}+t^2r^{-4}\1_{r\geq\sqrt t})\Big]\\
&\lesssim t^{-1}(\log t)^{-2a}e^{-\frac{r^2}{16t}} + |\Theta_*|^2 R^{-2a_1}t\1_{\{r\leq\sqrt t\}} + |\Theta_*|^2 R^{-2a_1}t^3r^{-4}e^{-\frac{r^2}{16t}}\1_{\{r\geq\sqrt t\}}\\
&\quad + |\Theta_*|^2 R^{-2a_1}t\1_{\{r\leq\sqrt t\}} + |\Theta_*|^2 R^{-2a_1}t^3r^{-4}\1_{\{r\geq\sqrt t\}}\\
& \lesssim (\log t)^{-\epsilon}|\Theta_*(\tau(t))|R^{-a_1}(\tau(t))(\1_{r\leq \sqrt t}+tr^{-2}\1_{r\geq\sqrt t}), 
\end{align*}

\begin{align*}
&\TT_6^{{\rm out}}\Big[\la^{2}|\Theta_*|^3 R^{18-3a}\langle \rho\rangle^{-16}\1_{\{\rho\sim R\}}\Big]\\
&\lesssim \TT_6^{{\rm out}}\Big[\la^{2}|\Theta_*|^3R^{2-3a}\1_{\{\frac{1}{2}\lambda R \leq r \leq 2\lambda R\}}\Big]\\
&\lesssim t^{-3}e^{-\frac{r^2}{16t}} + \la^{2}|\Theta_*|^3R^{2-3a}(\lambda R)^2\1_{\{r\leq 2\lambda R\}} + \la^{2}|\Theta_*|^3R^{2-3a}\lambda^6R^6|x|^{-4}e^{-\frac{r^2}{16t}}\1_{\{r\geq 2\lambda R\}} \\ 
& \lesssim (\log t)^{-\epsilon}|\Theta_*(\tau(t))|R^{-a_1}(\tau(t))(\1_{r\leq \sqrt t}+tr^{-2}\1_{r\geq\sqrt t}), 
\end{align*}

\begin{align*}
&\TT_6^{{\rm out}}\Big[\la^{2}\rho^2|\Theta|^3\1_{\{\rho\gtrsim R\}}\Big]\\
& \lesssim \TT_6^{{\rm out}}\Big[t^{-3}(\log t)^{-3a}r^2\1_{\{\lambda R\leq r\leq\sqrt t\}}+ r^{-4}(\log r)^{-3a}\1_{\{r\geq\sqrt t\}}\Big]\\
& \lesssim t^{-3}e^{-\frac{r^2}{16t}} + t^{-1}(\log t)^{-3a}\1_{r\leq \sqrt t} + t^{-2}(\log t)^{-3a}r^{-4}t^3e^{-\frac{r^2}{16t}}\1_{r\geq \sqrt t} \\
& \quad + t^{-1}(\log t)^{-3a}\1_{\{r\leq\sqrt t\}} + r^{-2}(\log r)^{-3a}tr^{-2}\1_{\{r\geq\sqrt t\}}\\
& \lesssim (\log t)^{-\epsilon}|\Theta_*(\tau(t))|R^{-a_1}(\tau(t))(\1_{r\leq \sqrt t}+tr^{-2}\1_{r\geq\sqrt t}),  
\end{align*}

\begin{align*}
&\TT_6^{{\rm out}}\Big[\la^{2}\rho^2\Big[|\Theta_*|^3 R^{-3a_1}(\1_{r\leq \sqrt t}+tr^{-2}\1_{r\geq\sqrt t})^3\Big]\1_{\{\rho\gtrsim R\}}\Big]\\
&\lesssim \TT_6^{{\rm out}}\Big[r^2|\Theta_*|^3 R^{-3a_1}(\1_{\lambda R \leq r\leq \sqrt t}+t^3r^{-6}\1_{r\geq\sqrt t})\Big]\\
&\lesssim t^{-1}(\log t)^{-3a}e^{-\frac{r^2}{16t}} + |\Theta_*|^3 R^{-3a_1}t^2\1_{\{r\leq\sqrt t\}} + |\Theta_*|^3 R^{-3a_1}t^4r^{-4}e^{-\frac{r^2}{16t}}\1_{\{r\geq\sqrt t\}}\\
&\quad + |\Theta_*|^3 R^{-3a_1}t^2\1_{\{r\leq\sqrt t\}} + |\Theta_*|^3 R^{-3a_1}t^4r^{-4}\1_{\{r\geq\sqrt t\}}\\
& \lesssim (\log t)^{-\epsilon}|\Theta_*(\tau(t))|R^{-a_1}(\tau(t))(\1_{r\leq \sqrt t}+tr^{-2}\1_{r\geq\sqrt t}).
\end{align*}

\noindent{\bf Claim 4}:
\begin{align*}
\TT_6^{{\rm out}}[g_4]\lesssim&~\|\psi\|_\flat\TT_6^{{\rm out}}\Big[\la^{-2}|\Theta_*|\langle\rho\rangle^{-4}R^{-a_1}(\1_{r\leq \sqrt t}+tr^{-2}\1_{r\geq\sqrt t})\1_{\{\rho\gtrsim R\}}\Big]+\TT_6^{{\rm out}}[\la^{-2}|\Theta|\langle\rho\rangle^{-4}\1_{\{\rho\gtrsim R\}}]\\
& \lesssim \left(\|\psi\|_\flat + 1\right)(\log t)^{-\epsilon}|\Theta_*(\tau(t))|R^{-a_1}(\tau(t))(\1_{r\leq \sqrt t}+tr^{-2}\1_{r\geq\sqrt t})\\
& \lesssim (\log t)^{-\epsilon/2}|\Theta_*(\tau(t))|R^{-a_1}(\tau(t))(\1_{r\leq \sqrt t}+tr^{-2}\1_{r\geq\sqrt t}).
\end{align*}
{\it Proof}: We have 
\begin{align*}
&\TT_6^{{\rm out}}[\la^{-2}|\Theta|\langle\rho\rangle^{-4}\1_{\{\rho\gtrsim R\}}]\\ 
& \lesssim \TT_6^{{\rm out}}\Big[t^{-1}(\log t)^{-a}\lambda^{2}r^{-4}\1_{\{\lambda R\leq r\leq\sqrt t\}}+ \lambda^{2}r^{-6}(\log r)^{-a}\1_{\{r\geq\sqrt t\}}\Big]\\
&\lesssim t^{\epsilon-2}(\log t)^{-a}e^{-\frac{|x|^2}{16t}} + t^{-1}(\log t)^{-a}R^{-2}\1_{r\leq \lambda R}+ t^{-1}(\log t)^{-a}\lambda^{2}r^{-2}\1_{\la R \leq r\leq\sqrt t}\\ &\quad +t^{-1}(\log t)^{-a}\lambda^{2}r^{-4}e^{-\frac{r^2}{16t}}t\1_{r\geq\sqrt t} + t^{\epsilon-2}\1_{r\leq\sqrt t} + \left(r^{-6}t^{1+\epsilon} + t^{\epsilon-2}e^{-\frac{|x|^2}{16t}}\right)\1_{r\geq\sqrt t}\\
& \lesssim (\log t)^{-\epsilon}|\Theta_*(\tau(t))|R^{-a_1}(\tau(t))(\1_{r\leq \sqrt t}+tr^{-2}\1_{r\geq\sqrt t}),
\end{align*}

\begin{align*}
&\TT_6^{{\rm out}}\Big[\la^{-2}|\Theta_*|\langle\rho\rangle^{-4}R^{-a_1}(\1_{r\leq \sqrt t}+tr^{-2}\1_{r\geq\sqrt t})\1_{\{\rho\gtrsim R\}}\Big]\\ 
& \lesssim \TT_6^{{\rm out}}\Big[t^{-1}(\log t)^{-a}\lambda^{2}R^{-a_1}r^{-4}(\1_{\{\lambda R\leq r\leq\sqrt t\}}+tr^{-2}\1_{r\geq\sqrt t})\Big]\\
&\lesssim t^{\epsilon-2}(\log t)^{-a}e^{-\frac{|x|^2}{16t}} + t^{-1}(\log t)^{-a}R^{-2-a_1}\1_{r\leq \lambda R}+ t^{-1}(\log t)^{-a}\lambda^{2}r^{-2}R^{-a_1}\1_{\la R \leq r\leq\sqrt t}\\ &\quad +t^{-1}(\log t)^{-a}\lambda^{2}r^{-4}R^{-a_1}e^{-\frac{r^2}{16t}}t\1_{r\geq\sqrt t} + t^{\epsilon-2}\1_{r\leq\sqrt t} + \left(r^{-6}t^{1+\epsilon} + t^{\epsilon-2}e^{-\frac{|x|^2}{16t}}\right)\1_{r\geq\sqrt t}\\
& \lesssim (\log t)^{-\epsilon}|\Theta_*(\tau(t))|R^{-a_1}(\tau(t))(\1_{r\leq \sqrt t}+tr^{-2}\1_{r\geq\sqrt t}).
\end{align*}

Now we solve the outer problem (\ref{eqn-outer1}) in the set 
$$
B_{out}: = \{\psi\ |\ \|\psi\|_{\flat} \leq C_0\}.
$$
The above estimates give us that $\TT_6^{{\rm out}}[\mathcal G[\psi, \phi, \lambda_1]]\in B_{out}$ for $\psi\in B_{out}$. It is a contraction mapping in $B_{out}$ by similar arguments, thus there exists a unique solution $\psi\in B_{out}$ and (\ref{psi-exists}) holds. The gradient estimate follows from the scaling arguments and standard parabolic regularity. 
\end{proof}

\medskip

\section{Solving the scaling parameter}\label{oscillating-scale}
To ensure the solvable of the inner problem (\ref{eqn-inner}), we need the following orthogonality condition 
\begin{equation}\label{orthog-full}
    \int_0^{4R} \left[(12U-6\rho^2 U^2)(\psi+\Theta)+\frac{4\la^{-1}\dot\la}{(1+\rho^2)^2} + \mathcal N[\la,\phi,\psi]\right]Z(\rho)\rho^5 d\rho =  0.
\end{equation}
\begin{proposition}
For $t_0$ sufficiently large, there exists a solution $\lambda _1 = \lambda_1[\phi]$ satisfying the following estimate:
\begin{equation*}
|\dot\lambda_1|\lesssim (\log t)^{-\epsilon}\lambda_0\Theta_*(\tau).
\end{equation*}
\end{proposition}
\begin{proof}
From the computations in Section \ref{GNT-formula}, we have 
\begin{equation}\label{theta-orthog}
\int_0^{4R} \left[(12U-6\rho^2 U^2)\Theta(\la\rho,t(\tau))+\frac{4\la^{-1}\dot\la}{(1+\rho^2)^2}\right]Z(\rho)\rho^5 d\rho = \sum_{i=1}^5I_i + 4\lambda^{-1}\dot\lambda\int_{0}^{4R}\frac{\rho^5}{(1+\rho^2)^4}d\rho.
\end{equation}
For the term $I_6: = \int_0^{4R}(12U-6\rho^2 U^2)\psi(\lambda\rho)Z(\rho)\rho^5 d\rho$, we have 
\begin{align*}
|I_6| & = \left|\int_0^{4R}(12U-6\rho^2 U^2)\psi(\lambda\rho)Z(\rho)\rho^5 d\rho\right| \leq 24\int_0^{4R}|\psi(\lambda\rho)|\frac{\rho^5}{(1+\rho^2)^4}d\rho\\
& \lesssim |\Theta_*(\tau(t))|R^{-a_1}(\tau(t))(\1_{r\leq \sqrt t}+tr^{-2}\1_{r\geq\sqrt t})\|\psi\|_\flat\int_0^{4R}\frac{\rho^5}{(1+\rho^2)^4}d\rho\lesssim \frac{1}{(\log t)^\epsilon}\Theta_*(\tau(t)),
\end{align*}
while for the term $I_7: = \int_0^{4R}\lambda^2\mathcal N[\la,\phi,\psi]Z(\rho)\rho^5 d\rho$, we have 
\begin{align*}
|I_7| & = \left|\int_0^{4R}\lambda^2\mathcal N[\la,\phi,\psi]Z(\rho)\rho^5 d\rho\right| \\
& \lesssim \left|\int_0^{4R}\left(|\Theta_*|^2 + |\Theta_*|^2R^{12-2a}\langle \rho\rangle^{-12} + |\Theta_*|^2R^{-2a_1}(\1_{r\leq \sqrt t}+tr^{-2}\1_{r\geq\sqrt t})^2\|\psi\|_\flat^2\right)Z(\rho)\rho^5 d\rho\right|\\
&\quad + \left|\int_0^{4R}\left(\lambda^2\rho^2|\Theta_*|^3 + \lambda^2\rho^2 |\Theta_*|^3R^{18-3a}\langle \rho\rangle^{-18} + \lambda^2\rho^2|\Theta_*|^3R^{-3a_1}(\1_{r\leq \sqrt t}+tr^{-2}\1_{r\geq\sqrt t})^3\|\psi\|_\flat^3\right)Z(\rho)\rho^5 d\rho\right|\\
& \lesssim |\Theta_*|^2\int_0^{4R}Z(\rho)\rho^5 d\rho + |\Theta_*|^2R^{12-2a}\int_0^{4R}Z(\rho)\langle \rho\rangle^{-7} d\rho + |\Theta_*|^2R^{-2a_1}\|\psi\|_\flat^2\int_0^{4R}Z(\rho)\rho^5 d\rho\\
&\quad+\lambda^2|\Theta_*|^3\int_0^{4R}Z(\rho)\rho^7 d\rho + \lambda^2 |\Theta_*|^3R^{18-3a}\int_0^{4R}Z(\rho)\langle \rho\rangle^{-13} d\rho + \lambda^2|\Theta_*|^3R^{-3a_1}\|\psi\|_\flat^3\int_0^{4R}Z(\rho)\rho^7 d\rho\\
& \lesssim |\Theta_*|^2R^2 + |\Theta_*|^2R^{12-2a} + |\Theta_*|^2R^{2-2a_1}\|\psi\|_\flat^2+\lambda^2|\Theta_*|^3R^4 + \lambda^2 |\Theta_*|^3R^{18-3a} + \lambda^2|\Theta_*|^3R^{4-3a_1}\|\psi\|_\flat^3\\
&\lesssim \frac{1}{(\log t)^\epsilon}\Theta_*(\tau(t)).
\end{align*}
Let us recall from Section \ref{GNT-formula} that in (\ref{theta-orthog}), 
\begin{align*}
    I_1=\left[C_1+O(R^{-2})+O(\la^2R^2/t)\right]\int_0^{M\sqrt t}\frac{s^5}{64t^3} e^{-\frac{s^2}{4t}}\Theta_0(s)ds,
\end{align*}
where $C_1=\int_0^\infty V(\rho) Z(\rho)\rho^5 d\rho=4$, $Z(\rho)=\frac1{(1+\rho^2)^2}$, $V(\rho)=\frac{24}{(1+\rho^2)^2}$. Also we have
\begin{align*}
    |I_3|& \lesssim t^{-1}(\log t)^{-\alpha}e^{-\frac{t}{\la^2 R^2}}(\frac{t}{\la^2 R^2}+1)+\la^{\epsilon}t^{-1-\epsilon/2}(\log t)^{-\alpha},\\
    |I_4|& \lesssim t^{-1-\epsilon/2}(\log t)^{-\alpha} \la^{\epsilon}R^{5/2+\epsilon},\\
    |I_5|& \lesssim C\lambda^2R^2\frac{1}{t^2}
\end{align*}
for $\epsilon>0$ small. Let us denote $\lambda(t) = \lambda_0(t) + \lambda_1(t)$. The main term $\lambda_0(t)$ of $\lambda(t)$ is 
$$
\int_0^{4R} \left((12U-6\rho^2 U^2)\Theta(\lambda\rho, t)+\frac{4\la^{-1}\dot\la}{(1+\rho^2)^2}\right)Z(\rho)\rho^5 d\rho  = 0.
$$
From the computations in Section \ref{GNT-formula}, the above identity reduces to 
$$
\frac{\dot\lambda_0}{\lambda_0} + C_1\int_{0}^{\frac{t}{2\la R}}\frac{s^5}{64t^3} e^{-\frac{s^2}{4t}}\Theta_0(s)ds = 0.
$$
Thus from Lemma \ref{choice-of-Theta_0}, $\lambda_0(t)$ takes the following form
\begin{align*}
&\la_0(t)= C_{t_0}\exp\left(-\frac{3}{4}\left(\log(2+t)\right)^{1-a}\cos(\log\log(2+t))\right),
\end{align*}
with $C_{t_0} = \exp\left(\frac{3}{4}\left(\log(2+t_0)\right)^{1-a}\cos(\log\log(2+t_0))\right)$.
Now we derive the equation for $\lambda_1(t)$ as follows,  
\begin{align}\label{equation-for-lambda_1}
\dot\lambda_1 + \lambda_0\sum_{i=2}^7 I_i + \lambda_0[O(R^{-2})+O(\la^2/t)]\int_0^{M\sqrt t}\frac{s^5}{8t^3} e^{-\frac{s^2}{4t}}\Theta_0(s)ds -\frac{\dot\la_0}{\la_0}\lambda -\frac{\dot\la_1}{\la_0}\lambda_1 = 0.
\end{align}
Let us work in the space 
$$
B_{\dot\lambda_1} : = \{f\in C[t_0, +\infty)\ |\ \|f\|_{\dot\lambda_1}\leq 1\}
$$
with the norm $\|f\|_{\dot\lambda_1}: = \sup_{t\geq t_0}(\log t)^\epsilon \left(\Theta_*(\tau)\right)^{-1}|f(t)|$. Then the Schauder fixed point theorem gives us a solution $\lambda_1\in B_{\dot\lambda_1}$ of (\ref{equation-for-lambda_1}),  and it holds that $\dot\lambda_1 \sim \frac{1}{(\log t)^\epsilon}\lambda_0\Theta_*(\tau)$.
\end{proof}

\section{Solving the inner problem: Proof of Theorem \ref{thm}}\label{inner-problem}
For the inner problem
\begin{equation*}
    \la^2 \phi_t=\phi_{\rho\rho}+\frac5\rho\phi_\rho+(12U-6\rho^2 U^2)\phi+\lambda^2(12U-6\rho^2 U^2)(\psi+\Theta)+\frac{4\la\dot\la}{(1+\rho^2)^2}+\lambda^2\mathcal N[\la,\phi,\psi]
\end{equation*}
in $B_{2R}\times (t_0,\infty)$ with
\begin{align*}
H[\phi](\rho, t) = \lambda^2(12U-6\rho^2 U^2)(\psi+\Theta)+\frac{4\la\dot\la}{(1+\rho^2)^2}+\lambda^2\mathcal N[\la,\phi,\psi].
\end{align*}
Let us consider the transformation of time variable  (\ref{time-variable-transformation}), then the inner problem becomes 
\begin{align}\label{inner-problem-time-new-variable}
    \phi_\tau & =\phi_{\rho\rho}+\frac5\rho\phi_\rho+(12U-6\rho^2 U^2)\phi\\
    &\qquad\qquad\qquad+\lambda^2(12U-6\rho^2 U^2)(\psi+\Theta)(\lambda \rho, t(\tau))+\frac{4\la\dot\la}{(1+\rho^2)^2}+\lambda^2\mathcal N[\la,\phi,\psi](\lambda \rho, t(\tau)).\notag
\end{align}
The right hand side of (\ref{inner-problem-time-new-variable}) can be estimated as follows,
\begin{align*}
&|\lambda^2(12U-6\rho^2 U^2)(\psi+\Theta)|\\
&\lesssim \lambda^2|\Theta_*|\langle\rho\rangle^{-4}R^{-a_1}(\1_{r\leq \sqrt t}+tr^{-2}\1_{r\geq\sqrt t})\|\psi\|_\flat + \lambda^2|\Theta_*|\langle\rho\rangle^{-4}\\
&\lesssim \lambda^2|\Theta_*|\langle \rho\rangle^{-2-a},
\end{align*}
\begin{align*}
&\left|\frac{4\la\dot\la}{(1+\rho^2)^2}\right| \lesssim \lambda^2|\Theta_*|\langle\rho\rangle^{-4}\lesssim \lambda^2|\Theta_*|\langle \rho\rangle^{-2-a},
\end{align*}
\begin{align*}
&|\lambda^2\mathcal N[\la,\phi,\psi]| \\
&\lesssim \lambda^2\left|6(U_\la+\Theta+\eta_R\la^{-2}\phi+\psi)^2-2r^2(U_\la+\Theta+\eta_R\la^{-2}\phi+\psi)^3-6U_\la^2+2r^2 U_\la^3\right. \\ &\qquad\qquad\qquad\qquad\qquad\qquad\qquad\qquad\qquad\qquad\qquad\qquad\qquad\qquad\qquad\left.- (12U_\la-6r^2 U_\la^2)(\Theta+\eta_R\la^{-2}\phi+\psi)\right|\\
&\lesssim \lambda^2(\Theta+\eta_R\la^{-2}\phi+\psi)^2 + \lambda^2r^2U_{\lambda}(\Theta+\eta_R\la^{-2}\phi+\psi)^2  + \lambda^2r^2(\Theta+\eta_R\la^{-2}\phi+\psi)^3\\
&\lesssim \lambda^2|\Theta_*|^2 + \lambda^2 |\Theta_*|^2R^{12-2a}\langle \rho\rangle^{-12}\|\phi\|_{\sharp}^2 + \lambda^2|\Theta_*|^2R^{-2a_1}(\1_{r\leq \sqrt t}+tr^{-2}\1_{r\geq\sqrt t})^2\|\psi\|_\flat^2\\
&\quad + \lambda^4\rho^2|\Theta_*|^3 + \lambda^4\rho^2 |\Theta_*|^3R^{18-3a}\langle \rho\rangle^{-18}\|\phi\|_{\sharp}^3 + \lambda^4\rho^2|\Theta_*|^3R^{-3a_1}(\1_{r\leq \sqrt t}+tr^{-2}\1_{r\geq\sqrt t})^3\|\psi\|_\flat^3\\
&\lesssim \left(|\Theta_*|R^{2+a} + |\Theta_*|R^{12-2a}\langle \rho\rangle^{-10+a}\|\phi\|_{\sharp}^2 + |\Theta_*|R^{2+a-2a_1}(\1_{r\leq \sqrt t}+tr^{-2}\1_{r\geq\sqrt t})^2\|\psi\|_\flat^2\right)\lambda^2|\Theta_*|\langle \rho\rangle^{-2-a}\\
&\quad + \left(\rho^{4+a}|\Theta_*|^2+\rho^2 |\Theta_*|^2R^{18-3a}\langle \rho\rangle^{-16+a}\|\phi\|_{\sharp}^3+\rho^{4+a}|\Theta_*|^2R^{-3a_1}(\1_{r\leq \sqrt t}+tr^{-2}\1_{r\geq\sqrt t})^3\|\psi\|_\flat^3\right)\lambda^4|\Theta_*|\langle \rho\rangle^{-2-a}\\
&\lesssim \left(|\Theta_*|R^{2+a} + |\Theta_*|R^{12-2a}\|\phi\|_{\sharp}^2 + |\Theta_*|R^{2+a-2a_1}(\1_{r\leq \sqrt t}+tr^{-2}\1_{r\geq\sqrt t})^2\|\psi\|_\flat^2\right)\lambda^2|\Theta_*|\langle \rho\rangle^{-2-a}\\
&\quad + \left(R^{4+a}|\Theta_*|^2+|\Theta_*|^2R^{20-3a}\|\phi\|_{\sharp}^3+|\Theta_*|^2R^{4+a-3a_1}(\1_{r\leq \sqrt t}+tr^{-2}\1_{r\geq\sqrt t})^3\|\psi\|_\flat^3\right)\lambda^4|\Theta_*|\langle \rho\rangle^{-2-a}.
\end{align*}
Now let us recall Lemma \ref{Prop-inner}, and by the choice of $\lambda$ in Section \ref{oscillating-scale}, the orthogonality condition (\ref{inner-orthogonal-cond}) holds, hence (\ref{inner-problem-time-new-variable}) is equivalent to the following fixed point problem
\begin{equation}\label{fixed-point-inner-problem}
\phi = \mathcal T_{in}[H[\phi]].
\end{equation}
From the above estimates, $\mathcal T_{in}[B_{in}]\subset B_{in}$ with $B_{in} = \{\phi: \|\phi\|_{\sharp}\leq C\}$. Furthermore, by similar arguments as in \cite{LiWeiZhangZhou}, the operator $\mathcal T_{in}[\phi]$ is compact, then the Schauder fixed point theorem ensures the existence of a fixed point for (\ref{fixed-point-inner-problem}), which is the solution of the inner problem (\ref{inner-problem-time-new-variable}). This concludes the proof of Theorem \ref{thm}.

\appendix

\section{Convolution estimates in the outer problem}\label{Appendix-convo}

\begin{proof}[Proof of Lemma \ref{linear-outer}]
{\bf (1)} We consider (\ref{RHS-with-upper-bound}). For $n\geq 3$, we have 
\begin{equation*}
\begin{aligned}
			&
			\TT_n^{out}\left[ v(t) |x|^{-b} (\log |x|)^{-a}\1_{\{ l_1(t) \le |x| \le l_2(t) \}}\right]\\
			&\lesssim
			t^{-\frac n2}
			\int_{\frac{t_{0}}{2} }^{\frac t2}
			\int_{\RR^n}
			e^{-\frac{|x- y|^2}{4 t} }
			v(s) |y|^{-b}(\log |y|)^{-a} \1_{\{ l_1(s) \le |y| \le l_2(s) \}} d y d s
			\\
			&\quad  +
			\sup\limits_{ t_1 \in [t/2,t] } v(t_1)
			\int_{\frac t2}^{t}
			\int_{\RR^n}
			(t-s)^{-\frac n2}
			e^{-\frac{|x- y|^2}{4(t-s )} }
			|y|^{-b} (\log |y|)^{-a} \1_{\{ C_l^{-1} l_1(t) \le |y| \le C_l l_2(t) \}} d y d s\\
			&:=
			u_{1} + \sup\limits_{ t_1 \in [t/2,t] } v(t_1) u_{2} .
\end{aligned}
\end{equation*}
For $|x|\le 2C_* t^{\frac 12}$, we have
\begin{equation*}
\begin{aligned}
		u_{1}&\lesssim
		t^{-\frac n2}
		\int_{\frac{t_{0}}{2} }^{\frac t2}
		\int_{\RR^n}
		v(s) |y|^{-b}(\log |y|)^{-a} \1_{\{ l_1(s)\le |y| \le l_2(s) \}} d y d s\\
		&\lesssim
		t^{-\frac n2}
		\int_{\frac {t_{0}}2} ^{\frac t2}
		v(s)
		\begin{cases}
			l_2^{n-b}(s)(\log(l_2(s)))^{-a}
			&
			\mbox{ \ if \ }
			b<n
			\\
			(\log(l_2(s)))^{1-a}
			&
			\mbox{ \ if \ }
			b=n
		\end{cases}
		d s.
\end{aligned}
\end{equation*}
For $|x| > 2C_* t^{\frac 12}$, one has  $|x-y|\ge \frac{|x|}{2}$. Then
	\begin{equation*}
		u_{1}
		\lesssim
		t^{-\frac n2} e^{-\frac{|x|^2}{16 t} }
		\int_{\frac {t_{0}}2} ^{\frac t2}
		v(s)
		\begin{cases}
			l_2^{n-b}(s)(\log(l_2(s)))^{-a}
			&
			\mbox{ \ if \ }
			b<n
			\\
			(\log(l_2(s)))^{1-a}
			&
			\mbox{ \ if \ }
			b=n
		\end{cases}
		d s
		.
	\end{equation*}
For $|x| \le (2 C_l)^{-1} l_1(t)$,
since $\frac{|y|}{2} \le |x-y|\le 2|y|$, then for a sufficiently small $
\epsilon > 0$, we have 
	\begin{equation*}
		\begin{aligned}
			u_{2}
			&\le
			\int_{\frac t2}^{t}
			\int_{\RR^n}
			(t-s)^{-\frac n2}
			e^{-\frac{|y|^2}{16 (t-s )} }
			|y|^{-b} (\log |y|)^{-a} \1_{\{ C_l^{-1} l_1(t) \le |y| \le C_l l_2(t) \}} d y d s\\
			&\lesssim
            \begin{cases}
			l_2(t)^{\epsilon} (\log l_2(t))^{-a}\int_{\frac t2}^{t}
			(t-s)^{-\frac{b+\epsilon}{2}}
			\int_{\frac{l_1^2(t)}{16 C_l^2 (t-s)} }^{
				\frac{C_l^2 l_2^2(t)}{16(t-s)}}
			e^{-z} z^{\frac{n-b-\epsilon}{2} -1}  d z d s& \mbox{ \ if \ }
				b<2
				\\
                (\log l_1(t))^{-a}\int_{\frac t2}^{t}
			(t-s)^{-\frac b2}
			\int_{\frac{l_1^2(t)}{16 C_l^2 (t-s)} }^{
				\frac{C_l^2 l_2^2(t)}{16(t-s)}}
			e^{-z} z^{\frac{n-b}{2} -1}  d z d s& \mbox{ \ if \ }
				b=2\\
                (\log l_1(t))^{-a}\int_{\frac t2}^{t}
			(t-s)^{-\frac{b-\epsilon}{2}}
			\int_{\frac{l_1^2(t)}{16 C_l^2 (t-s)} }^{
				\frac{C_l^2 l_2^2(t)}{16(t-s)}}
			e^{-z} z^{\frac{n-b+\epsilon}{2} -1}  d z d s& \mbox{ \ if \ }
				b>2
			\end{cases}
            \\
            &\lesssim\begin{cases}
			  l_2(t)^{\epsilon} (\log l_2(t))^{-a}
			\left(
			\int_{\frac t2}^{t-l_2^2(t) }
			+
			\int_{t-l_2^2(t)}^{t-l_1^2(t)}
			+
			\int_{t- l_1^2(t) }^{t}
			\cdots
			\right)
			:=  l_2(t)^{\epsilon} (\log l_2(t))^{-a}
			\left(u_{21}^1 + u_{22}^1 + u_{23}^1\right) \mbox{ \ if \ }
				b<2\\
            (\log l_1(t))^{-a}\left(
			\int_{\frac t2}^{t-l_2^2(t) }
			+
			\int_{t-l_2^2(t)}^{t-l_1^2(t)}
			+
			\int_{t- l_1^2(t) }^{t}
			\cdots
			\right)
			:=  (\log l_1(t))^{-a}
			\left(u_{21}^2 + u_{22}^2 + u_{23}^2\right) \mbox{ \ if \ }
				b=2\\
            (\log l_1(t))^{-a}\left(
			\int_{\frac t2}^{t-l_2^2(t) }
			+
			\int_{t-l_2^2(t)}^{t-l_1^2(t)}
			+
			\int_{t- l_1^2(t) }^{t}
			\cdots
			\right)
			:=  (\log l_1(t))^{-a}
			\left(u_{21}^3 + u_{22}^3 + u_{23}^3\right) \mbox{ \ if \ }
				b>2\\
            \end{cases}\\
			& \lesssim
			\begin{cases}
				l_2^{2-b}(t)(\log l_2(t))^{-a}
				& \mbox{ \ if \ }
				b<2
				\\
			\langle	\ln (\frac{l_2(t)}{l_1(t)}) \rangle  (\log l_1(t))^{-a}
				& \mbox{ \ if \ }
				b=2
				\\
			l_1^{2-b}(t)(\log l_1(t))^{-a}
				& \mbox{ \ if \ }
				b>2 .
			\end{cases}
		\end{aligned}
	\end{equation*}
To obtain the above inequality, we have the following estimates (with $b$ in place of $b+\epsilon$ if $b<2$). Since we assume $n\geq 3$, it holds that 
	\begin{equation*}
			\int_{\frac t2}^{t- l_2^2(t) }
			(t-s)^{-\frac b2}
			\int_{\frac{l_1^2(t) }{ 16 C_l^2 (t-s)}}^{\frac{ C_l^2 l_2^2(t) }{16(t-s)}}
			z^{\frac{n-b}{2} - 1}
			d z d s
			\lesssim
			\begin{cases}
				l_2^{2-b}(t)
				&
				\mbox{ \ if \ } b<n
				\\
				\langle \log (\frac{l_2(t)}{l_1(t)}) \rangle
				l_2^{2-n}(t)
				&
				\mbox{ \ if \ } b=n.
			\end{cases}
	\end{equation*}
Since $l_1^2(t) \lesssim t-s \lesssim l_2^2(t)$, then
	\begin{equation*}
		\begin{aligned}
			\int_{t-l_2^2(t)}^{t-l_1^2(t)}
			(t-s)^{-\frac b2}
			\int_{\frac{l_1^2(t) }{ 16 C_l^2 (t-s)}}^{\frac{ C_l^2 l_2^2(t) }{16(t-s)}}
			e^{-z}z^{\frac{n-b}{2} - 1}
			d z d s & \lesssim\int_{t-l_2^2(t)}^{t-l_1^2(t)}
			(t-s)^{-\frac b2}
			\begin{cases}
				1
				& \mbox{ \ if \ }
				b<n
				\\
			\langle \log ( \frac{ t-s}{l_1^2(t)})  \rangle
				& \mbox{ \ if \ }
				b=n
			\end{cases}
			d s\\
			&\lesssim
			\begin{cases}
				l_2^{2-b}(t)
				& \mbox{ \ if \ }
				b<2
				\\
				\langle \log (\frac{l_2(t)}{l_1(t)})  \rangle
				& \mbox{ \ if \ }
				b=2
				\\
				l_1^{2-b}(t)
				& \mbox{ \ if \ }
				b>2 .
			\end{cases}
		\end{aligned}
	\end{equation*}	
Also we have
	\begin{equation*}
    \begin{aligned}
			\int_{t-l_1^2(t)}^{t}(t-s)^{-\frac b2}
			\int_{\frac{l_1^2(t) }{ 16 C_l^2 (t-s)}}^{\frac{ C_l^2 l_2^2(t) }{16(t-s)}}
			e^{-z}z^{\frac{n-b}{2} - 1}
			d z d s & \lesssim\int_{t-l_1^2(t)}^{t}
			(t-s)^{-\frac b2}
			\int_{\frac{l_1^2(t) }{ 16 C_l^2 (t-s)}}^{\frac{C_l^2 l_2^2(t) }{16(t-s) }}
			e^{-\frac z2}
			d z d s\\
			&\lesssim
			\int_{t-l_1^2(t)}^{t}
			(t-s)^{-\frac b2}
			e^{- \frac{l_1^2(t)   }{ 32C_l^2 (t-s)} } d s
			\sim
			l_1^{2-b}(t) .
    \end{aligned}
	\end{equation*}
For $ (2C_l)^{-1} l_1(t) \le |x| \le 2C_l l_2(t)$,  then
	\begin{equation*}
		\begin{aligned}
			u_{2} & \le\int_{\frac t2}^{t}
			\int_{\RR^n} 
			(t-s)^{-\frac n2}
			e^{-\frac{|x- y|^2}{4(t-s )} } 
			|y|^{-b}(\log |y|)^{-a}\\
			&\qquad\qquad\qquad\qquad\qquad\qquad\qquad\left(
			\1_{\{ (4C_l)^{-1} l_1(t) \le |y| \le \frac{|x|}{2} \}}
			+
			\1_{\{\frac{|x|}{2} \le |y| \le 2|x| \}}
			+
			\1_{\{ 2|x| \le |y| \le 4C_l l_2(t) \}}
			\right)
			d y d s
			\\
			& := 
			u_{21} + u_{22} + u_{23}
			\lesssim
			\begin{cases}
				l_2^{2-b}(t)\langle\log l_2(t)\rangle^{-a}
				&
				\mbox{ \ if \ } b<2
				\\
				\langle \log(\frac{l_2(t)}{|x|} ) \rangle \langle\log |x|\rangle^{-a}
				&
				\mbox{ \ if \ } b=2
				\\
				|x|^{2-b}\langle\log |x|\rangle^{-a}
				&
				\mbox{ \ if \ } 2<b<n
                \\
				|x|^{2-n} \langle	\log (\frac{|x|}{l_1(t)}) \rangle\langle\log l_1(t)\rangle^{-a}
				&
				\mbox{ \ if \ } b=n.
			\end{cases}
		\end{aligned}
	\end{equation*}
For $u_{21}$, since $n\geq 3$, one has
	\begin{equation*}
    \begin{aligned}
			u_{21}
			&\leq
			\int_{\frac t2}^{t}
			\int_{\RR^n}
			(t-s)^{-\frac n2}
			e^{-\frac{|x|^2}{16(t-s )} }
			|y|^{-b}(\log |y|)^{-a}
			\1_{ \{ (4C_l)^{-1} l_1(t) \le |y| \le \frac{|x|}{2} \}}
			d y d s\\
			& \lesssim
			\begin{cases}
				|x|^{2-b}\langle \log |x|\rangle ^{-a}
				&
				\mbox{ \ if \ } b<n
				\\
				|x|^{2-n} \langle	\log (\frac{|x|}{l_1(t)}) \rangle\langle\log l_1(t)\rangle^{-a}
				&
				\mbox{ \ if \ } b=n.
			\end{cases}
    \end{aligned}
	\end{equation*}
	For $u_{22}$, we have
	\begin{equation*}
    \begin{aligned}
			u_{22}
		    &\lesssim
			|x|^{-b}\langle\log (|x|)\rangle^{-a}
			\int_{\frac t2}^{t}
			\int_{\RR^n}
			(t-s)^{-\frac n2}
			e^{-\frac{|x- y|^2}{4(t-s )} }
			\1_{\{ |x-y| \le 3|x| \}}
			d y d s
			\\
            &\lesssim
			|x|^{-b}\langle\log (|x|)\rangle^{-a}
			\int_{\RR^n}\frac{1}{|x-y|^{n-2}} \1_{\{ |x-y| \le 3|x| \}}d y
			\sim
			|x|^{2-b}\langle\log (|x|)\rangle^{-a}.
    \end{aligned}
	\end{equation*}
	For $u_{23}$, we have
	\begin{equation*}
		\begin{aligned}
			u_{23}
			& \le
			\int_{\frac t2}^{t}
			\int_{\RR^n}
			(t-s)^{-\frac n2}
			e^{-\frac{|y|^2}{16 (t-s )} }
			|y|^{-b}(\log |y|)^{-a}
			\1_{\{ 2|x| \le |y| \le 4C_l l_2(t) \}}
			d y d s\\
			& = 
			\left(
			\int_{\frac t2}^{t-\frac{l_2^2(t)}{2C_*^2}}
			+
			\int_{ t-\frac{l_2^2(t)}{2C_*^2} }^{t- \frac{|x|^2}{ 8C_l^2 C_*^2 }}
			+
			\int_{t- \frac{|x|^2}{8C_l^2 C_*^2 }}^{t}
			\cdots
			\right)
			:=
			u_{231} + u_{232} + u_{233}\\
			&\lesssim
			\begin{cases}
				l_2^{2-b}(t)\langle\log l_2(t)\rangle^{-a}
				&
				\mbox{ \ if \ } b<2
				\\
				\langle \log(\frac{l_2(t)}{|x|} ) \rangle \langle\log |x|\rangle^{-a}
				&
				\mbox{ \ if \ } b=2
				\\
				|x|^{2-b} \langle\log |x|\rangle^{-a}
				&
				\mbox{ \ if \ } b>2 .
			\end{cases}
		\end{aligned}
	\end{equation*}
Indeed, for $u_{231}$, we have
	\begin{equation*}
			u_{231}
			\lesssim
			\begin{cases}
				l_2^{2-b}(t)\langle\log l_2(t)\rangle^{-a}
				&
				\mbox{ \ if \ } b<n
				\\
				l_2^{2-n}(t)
			\langle\log \frac{l_2(t)}{|x|}\rangle\langle\log |x| \rangle^{-a}
				&
				\mbox{ \ if \ } b=n.
			\end{cases}
	\end{equation*}
	For $u_{232}$, since $n\geq 3$, we estimate
	\begin{equation*}
    \begin{aligned}
			u_{232}
			&\lesssim
	\int_{ t-\frac{l_2^2(t)}{2C_*^2} }^{t- \frac{|x|^2}{ 8C_l^2 C_*^2 }}
			(t-s)^{-\frac b2}
			\begin{cases}
				(t-s)^{\frac{\epsilon}{2}}(l_2(t))^\epsilon\langle\log l_2(t) \rangle^{-a}
				&
				\mbox{ \ if \ } b<n \mbox{\ and } b\neq 2
				\\
                \langle\log |x| \rangle^{-a}
				&
				\mbox{ \ if \ } b = 2\\
			\langle \log(\frac{|x|^2}{t-s}) \rangle\langle \log |x|\rangle^{-a}
				&
				\mbox{ \ if \ } b=n
			\end{cases}
			d s\\
			&\lesssim
			\begin{cases}
				l_2^{2-b}(t)\langle\log l_2(t) \rangle^{-a}
				&
				\mbox{ \ if \ } b<2
				\\
				\langle	\log(\frac{l_2(t)}{|x|} ) \rangle\langle\log |x|\rangle^{-a}
				&
				\mbox{ \ if \ } b=2
				\\
				|x|^{2-b}\langle \log |x|\rangle^{-a}
				&
				\mbox{ \ if \ } b>2 .
			\end{cases}
    \end{aligned}
	\end{equation*}
	For $u_{233}$, one has
	\begin{equation*}
			u_{233}
			\lesssim
			\int_{t- \frac{|x|^2}{8C_l^2 C_*^2 }}^{t}
			(t-s)^{-\frac b2}
			e^{-\frac{|x|^2}{8(t-s)}} d s \langle \log |x|\rangle^{-a}
			\sim
			|x|^{2-b}\langle \log |x|\rangle^{-a}
			\int_{ C_l^2 C_{*}^2}^{\infty}
			e^{-z} z^{\frac b2 -2}
			d z
			\sim
			|x|^{2-b}\langle \log |x|\rangle^{-a}.
	\end{equation*}
For $|x|\ge 2C_l l_2(t)$, since $\frac{|x|}{2} \le |x-y| \le 2|x|$, then for $n\geq 3$, we have 
	\begin{align*}
			u_{2}
			&\lesssim
			\int_{\frac{t}{2} }^{t}  (t-s)^{-\frac n2}
			e^
			{-\frac{|x|^2}{16(t-s )} }
			d s
			\begin{cases}
				l_2^{n-b}(t)\langle	\log l_2(t)\rangle^{-a}
				&
				\mbox{ \ if \ } b<n
				\\
			\langle	\log l_2(t)\rangle^{1-a}
				&
				\mbox{ \ if \ } b=n
			\end{cases}\\
			&\lesssim
			|x|^{2-n}
			e^{-\frac{|x|^2}{16 t}}
			\begin{cases}
				l_2^{n-b}(t)\langle	\log l_2(t)\rangle^{-a}
				&
				\mbox{ \ if \ } b<n
				\\
				\langle	\log l_2(t)\rangle^{1-a}
				&
				\mbox{ \ if \ } b=n.
			\end{cases}
	\end{align*}
{\bf (2)} Now we consider the estimate (\ref{RHS-without-upper-bound}). Write
		\begin{align*}
			&
			\TT_{n}^{out}\left[v(t) |x|^{-b}(\log |x|)^{-a} \1_{\{ |x|\ge t^{\frac 12}\}}\right]\\
			&\lesssim
			t^{-\frac n2}
			\int_{t_{0}/2}^{t/2}
			\int_{\RR^n}
			e^{-\frac{|x- y|^2}{4 t} }
			v(s)|y|^{-b}(\log |y|)^{-a} \1_{\{ |y|\ge s^{\frac 12}\}} d y d s
			\\
			&\quad +
			\sup\limits_{ t_1 \in [t/2,t]} v(t_1)
			\int_{t/2}^t
			\int_{\RR^n}
			(t-s)^{-\frac n2}
			e^{-\frac{|x- y|^2}{4(t-s )} }
			|y|^{-b}(\log |y|)^{-a} \1_{\{ |y|\ge 2^{-\frac 12} t^{\frac 12}\}} d y d s\\
			& :=
		t^{-\frac n2} u_{1} + \sup\limits_{ t_1 \in [t/2,t] } v(t_1) u_{2}.
		\end{align*}
	For $u_{1}$, when $|x|\le 2t^{\frac 12}$, we have
	\begin{equation*}
		\begin{aligned}
		u_{1}
			& \lesssim
			\int_{t_{0}/2}^{t/2}
			\int_{\RR^n}
			v(s)|y|^{-b}(\log |y|)^{-a}
			\1_{\{ s^{\frac 12} \le |y|\le 4 t^{\frac 12}  \}}
			d y d s
			+
			\int_{t_{0}/2}^{t/2}
			\int_{\RR^n}
			e^{-\frac{|y|^2}{16 t} }
			v(s)|y|^{-b}(\log |y|)^{-a}
			\1_{\{ 4 t^{\frac 12} \le  |y| \}}
			d y d s
			\\
&\lesssim
\int_{t_{0}/2}^{t/2}
v(s)
\begin{cases}
	t^{\frac{n-b}{2}}(\log(t))^{-a}
	&
	\mbox{ \ if \ } b<n
	\\
	(\log(t))^{1-a}
	&
	\mbox{ \ if \ } b=n
\end{cases}
d s
+
t^{\frac {n-b}2}(\log(t))^{-a}
\int_{t_{0}/2}^{t/2} v(s) d s\\
&\sim
\int_{t_{0}/2}^{t/2}
v(s)
\begin{cases}
	t^{\frac{n-b}{2}}(\log(t))^{-a}
	&
	\mbox{ \ if \ } b<n
	\\
	(\log(t))^{1-a}
	&
	\mbox{ \ if \ } b=n
\end{cases}
d s  .
		\end{aligned}
	\end{equation*}
For $u_{1}$, when $|x| > 2 t^{\frac 12}$, we have
		\begin{align*}
			u_{1} = \ &
			\int_{t_{0}/2}^{t/2}
			\int_{\RR^n}
			e^{-\frac{|x- y|^2}{4 t} }
			v(s)|y|^{-b}(\log |y|)^{-a}
			\left(
			\1_{\{ s^{\frac 12} \le |y|\le \frac{|x|}{2} \}}
			+
			\1_{\{ \frac{|x|}{2} \le |y|\le 2|x|\}}
			+
			\1_{\{ 2|x| \le |y| \}}
			\right)
			d y d s
			\\
\lesssim \ &
	\int_{t_{0}/2}^{t/2}
	v(s)
\int_{\RR^n}
\left(
e^{-\frac{|x|^2}{16 t} }
|y|^{-b}(\log |y|)^{-a}
\1_{\{ s^{\frac 12} \le |y|\le \frac{|x|}{2} \}}
+
|x|^{-b}(\log |x|)^{-a}
e^{-\frac{|x- y|^2}{4 t} }
\1_{\{ |x-y|\le 3|x|\}}
\right)
d y d s
\\  
&\quad + \int_{t_{0}/2}^{t/2}
	v(s)
\int_{\RR^n}
\left(
e^{-\frac{ |y|^2}{16 t} }
|y|^{-b}(\log |y|)^{-a}
\1_{\{ 2|x| \le |y| \}}
\right)
d y d s
\\
\lesssim \ &
 e^{-\frac{|x|^2}{16 t} }
			\int_{t_{0}/2}^{t/2}
			v(s)
			\begin{cases}
				|x|^{n-b}(\log(|x|))^{-a}
				&
				\mbox{ \ if \ }
				b<n
				\\
			(\log(|x|))^{1-a}
				&
				\mbox{ \ if \ }
				b=n
			\end{cases}
			d s
			+ t^{\frac{n}{2}} |x|^{-b}(\log(|x|))^{-a}
			\int_{t_{0}/2}^{t/2}
			v(s) d s\\
            & +
			t^{\frac {n-b}2}(\log(|x|))^{-a} e^{-\frac{|x|^2}{8t}}
			\int_{t_{0}/2}^{t/2}
			v(s) d s
			\\
			\lesssim \ &
	 e^{-\frac{|x|^2}{16 t} }
			\int_{t_{0}/2}^{t/2}
			v(s)
			\begin{cases}
				0
				&
				\mbox{ \ if \ }
				b<n
				\\
			(\log(|x|))^{1-a}
				&
				\mbox{ \ if \ }
				b=n
			\end{cases}
			d s
			+ 	t^{ \frac n2} |x|^{-b}(\log(|x|))^{-a}
			\int_{t_{0}/2}^{t/2}
			v(s) d s .
		\end{align*}
For $u_{2}$, when $|x|\le 2^{-\frac 32} t^{\frac 12}$, we have $|y|\ge 2|x|$. Then
	\begin{equation*}
    \begin{aligned}
			u_{2} &= 
			\int_{t/2}^t
			\int_{\RR^n}
			(t-s)^{-\frac n2}
			e^{-\frac{|y|^2}{16(t-s )} }
			|y|^{-b}(\log |y|)^{-a} \1_{\{ |y|\ge 2^{-\frac{1}{2}} t^{\frac 12}\}} d y d s\\
			&\lesssim
			(\log t)^{-a}\int_{t/2}^t
			(t-s)^{-\frac b2}
			e^{- \frac{t}{64(t-s) }}
			d s
			\sim
			t^{1-\frac b2} (\log t)^{-a}.
    \end{aligned}
	\end{equation*}
For $u_{2}$, when $|x|\ge 2^{-\frac 32} t^{\frac 12}$, one has
		\begin{align*}
			&u_{2} \\
            & =
			\int_{t/2}^t
			\int_{\RR^n}
			(t-s)^{-\frac n2}
			e^{-\frac{|x- y|^2}{4(t-s )} }
			|y|^{-b} (\log |y|)^{-a}
			\left(
			\1_{\{ 9^{-1} t^{\frac 12} \le |y|\le  \frac{|x|}{2} \}}
			+
			\1_{\{ \frac{|x|}{2} \le |y|\le  4|x| \}}
			+
			\1_{\{ 4|x|\le |y| \}}
			\right) d y d s
			\\
&\lesssim
\int_{t/2}^t
\int_{\RR^n}
(t-s)^{-\frac n2}
\left(
e^{-\frac{|x|^2}{16(t-s )} }
|y|^{-b}(\log |y|)^{-a}
\1_{\{ 9^{-1} t^{\frac 12} \le |y|\le  \frac{|x|}{2} \}}
+
|x|^{-b}(\log |x|)^{-a} e^{-\frac{|x- y|^2}{4(t-s )} }
\1_{\{ |x-y|\le  5|x| \}}
\right) d y d s
\\
&\quad +\int_{t/2}^t
\int_{\RR^n}
(t-s)^{-\frac n2}
	e^{-\frac{|y|^2}{16(t-s )} }
|y|^{-b}(\log |y|)^{-a}
\1_{\{ 4|x|\le |y| \}} d y d s
\\
			& \lesssim
			|x|^{2-n}
			e^{-\frac{|x|^2}{16 t}}
			\begin{cases}
				|x|^{n-b}(\log |x|)^{-a}
				&
				\mbox{ \ if \ } b<n
				\\
		        (\log(\frac{|x|}{\sqrt{t}}))(\log t)^{-a}
				&
				\mbox{ \ if \ } b=n
			\end{cases}
			+
			t  |x|^{-b}(\log |x|)^{-a}
			+
			|x|^{2-b}(\log |x|)^{-a}
			e^{-\frac{|x|^2}{2t}}\\
			&\sim
			t |x|^{-b}(\log t)^{-a}.
		\end{align*}
This concludes the proof.
\end{proof}

\medskip

\begin{proof}[Proof of Lemma \ref{Lemma2.32.3}]
Set
\begin{align*}
|\Theta(r,t)| & \lesssim (4\pi t)^{-\frac n2} \int_{\RR^n}e^{-\frac{|x-y|^2}{4t} }\langle y \rangle^{-b}(\log \langle y \rangle)^{-a}d y\\
& \sim t^{-\frac n2}\left(\int_{|y|\le \frac{|x|}{2}}
+
\int_{\frac{|x|}{2} \le|y| \le 2|x|}
+
\int_{2|x| \le |y|}
\right)
e^{-\frac{|x-y|^2}{4t} }
\langle y \rangle^{-b}(\log \langle y \rangle)^{-a} d y.
\end{align*}
We estimate the above terms as follows,
\begin{align*}
\int_{|y|\le \frac{|x|}{2}}
		e^{-\frac{|x-y|^2}{4t} }
		\langle y\rangle^{-b}(\log \langle y \rangle)^{-a} d y
&\lesssim e^{-\frac{|x|^2}{16t} }
		\int_{|y|\le \frac{|x|}{2}}
		\langle y\rangle^{-b}(\log \langle y \rangle)^{-a}  d y\\
		&\lesssim
\begin{cases}
e^{-\frac{|x|^2}{16t} } |x|^{n}
\mbox{ \ if \ } |x| \le  1
\\
		e^{-\frac{|x|^2}{16t} } |x|^{n-b} (\ln (|x| +2 ))^{-a}
&
\mbox{ \ if \ } |x|> 1
\end{cases},
\end{align*}
\begin{align*}
			&\int_{\frac{|x|}{2} \le|y| \le 2|x|}
			e^{-\frac{|x-y|^2}{4t} }
			\langle y\rangle^{-b}(\log \langle y \rangle)^{-a}  d y\\
            &\lesssim
			\langle x\rangle^{-b}(\log \langle x \rangle)^{-a}
			\int_{ |x-y| \le 3|x|}
			e^{-\frac{|x-y|^2}{4t} }
			d y
			\sim
			\begin{cases}
			 \langle x\rangle^{-b}(\log \langle x \rangle)^{-a} 	|x|^{n}
				&
				\mbox{ \ if \ } |x|\le t^{\frac{1}{2}}
				\\
			\langle x\rangle^{-b}(\log \langle x \rangle)^{-a}
			t^{\frac{n}{2}}
				&
				\mbox{ \ if \ } |x| > t^{\frac{1}{2}}
			\end{cases},
\end{align*}
and
\begin{equation*}
\int_{2|x| \le |y|}
e^{-\frac{|x-y|^2}{4t} }
\langle y\rangle^{-b} (\log \langle y \rangle)^{-a}d y
\le
\int_{2|x| \le |y|}
e^{-\frac{|y|^2}{16t} }
\langle y\rangle^{-b}(\log \langle y \rangle)^{-a} d y .
\end{equation*}
For $|x|\ge 1$, we have
\begin{align*}
&\int_{2|x| \le |y|}
			e^{-\frac{|y|^2}{16t} }
			\langle y\rangle^{-b}(\log \langle y \rangle)^{-a} d y
			\lesssim\\
			  &
			\begin{cases}
				\int_{2|x| \le |y|\leq 2\sqrt{t}}
			e^{-\frac{|y|^2}{16t} }
			\langle y\rangle^{-b}(\log \langle y \rangle)^{-a} d y+\int_{2\sqrt{t} \le |y|}
			e^{-\frac{|y|^2}{16t} }
			\langle y\rangle^{-b}(\log \langle y \rangle)^{-a} d y\lesssim t^{\frac{n-b}{2}}(\log \langle t \rangle)^{-a}
				&
				\mbox{ \ if \ }
				|x|\le t^{\frac 12},
				\\
				t^{\frac{n-b}{2}}(\log \langle |x| \rangle)^{-a}		
			    \int_{\frac{|x|^2}{4t}}^\infty
			    e^{-z} z^{\frac{n-b}{2}-1} d z\lesssim t^{\frac{n-b}{2}}(\log \langle |x|\rangle)^{-a}
				e^{-\frac{|x|^2}{8t}}
				&
				\mbox{ \ if \ }
				|x| > t^{ \frac 12 }
			\end{cases} .
	\end{align*}
For $|x|<1$, we have
\begin{equation*}
	\begin{aligned}
		&
		\int_{2|x| \le |y|}
		e^{-\frac{|y|^2}{16t} }
		\langle y\rangle^{-b}(\log \langle y \rangle)^{-a} d y
		\sim
\int_{2|x|}^{2}
	e^{-\frac{r^2}{16t} } r^{n-1} dr
+
		\int_{2}^\infty
		e^{-\frac{ r^2}{16t} }
		r^{n-1-b}(\log r)^{-a} d r
\\
& \lesssim
\begin{cases}
	t^{\frac n2} e^{ -\frac{|x|^2}{8t} }
& \mbox{ \ if \ } t\le |x|^2
\\
	t^{\frac n2}
& \mbox{ \ if \ } |x|^2<t\le 1
\\
1 & \mbox{ \ if \ } t >1
\end{cases}
+
\begin{cases}
t^{\frac{n-b}{2}}
e^{-\frac{ 1 }{8t}}\lesssim t^{\frac{n-b}{2}}
e^{-\frac{ |x|^2 }{8t}}
&
\mbox{ \ if \ }
t<1
\\
	t^{\frac{n-b}{2}}(\log(t+2))^{-a}
	&
	\mbox{ \ if \ }
 t\ge 1
\end{cases} \\
&\lesssim
\begin{cases}
t^{\frac n2} e^{ -\frac{|x|^2}{16 t} }
& \mbox{ \ if \ } t\le |x|^2
\\
	t^{\frac n2}
& \mbox{ \ if \ } |x|^2<t\le 1
	\\
	t^{\frac{n-b}{2}}(\log(t+2))^{-a}
	&
	\mbox{ \ if \ }
	t\ge 1
\end{cases} .
	\end{aligned}
\end{equation*}
Combining above estimates, one has
\begin{equation*}
	\begin{aligned}
		|\Theta(r,t)|
		\lesssim \ &
		\begin{cases}
			\langle t \rangle^{-\frac{b}{2}}(\log(t+2))^{-a}	
			&
			\mbox{ \ if \ }
			|x|\le \max\{1,t^{\frac 12} \},
			\\
			\langle x\rangle^{-b}(\ln (|x|+2))^{-a}  
			&
			\mbox{ \ if \ }
			|x| > \max\{1,t^{\frac 12} \}.
		\end{cases}
	\end{aligned}
\end{equation*}
For the gradient estimate (\ref{Estimate-for-Theta-Gradient}), we have 
\begin{align*}
&\left|\nabla_r \Theta(r,t)\right| \\
& \lesssim(4\pi t)^{-\frac n2} \left|\int_{\RR^n}\left(-\frac{x-y}{2t}\right)e^{-\frac{|x-y|^2}{4t} }\langle y \rangle^{-b}(\log \langle y \rangle)^{-a}d y\right|\\
& \lesssim(4\pi t)^{-\frac n2} \left|-\frac{x}{2t}\right|\int_{\RR^n}e^{-\frac{|x-y|^2}{4t} }\langle y \rangle^{-b}(\log \langle y \rangle)^{-a}d y+ (4\pi t)^{-\frac n2} \left(\frac{1}{2t}\right)\int_{\RR^n}|y|e^{-\frac{|x-y|^2}{4t} }\langle y \rangle^{-b}(\log \langle y \rangle)^{-a}d y\\
& = \left|-\frac{x}{2t}\right| \left|\Theta(x, t)\right| + (4\pi t)^{-\frac n2} \left(\frac{1}{2t}\right)\int_{\RR^n}|y|e^{-\frac{|x-y|^2}{4t} }\langle y \rangle^{-b}(\log \langle y \rangle)^{-a}d y\\
& \lesssim
\begin{cases}
\langle t \rangle^{-\frac{b}{2}}(\log(t+2))^{-a}\left(\frac{1}{\sqrt{t}} + \frac{|x|}{t}\right)
&
\mbox{ \ if \ }
|x|\le \max\{1,t^{\frac 12} \},
\\
\langle x\rangle^{-b}(\ln (|x|+2))^{-a}\frac{|x|}{t}
&
\mbox{ \ if \ }
|x| > \max\{1,t^{\frac 12} \}.
\end{cases}
\end{align*}
This completes the proof.
\end{proof}

\medskip

\section{Radial heat kernel}\label{radial-heat-kernel}
Let us consider the following equation
\begin{align*}
u_t = u_{rr} + \frac{n-1}{r}u_r,
\end{align*}
set $u = r^{\frac{2-n}{2}}v(r, t)$, then we get 
\begin{align*}
v_t = v_{rr} + \frac{1}{r}v_r + \frac{\alpha(\alpha+n-2)}{r^2}v,
\end{align*}
with $\alpha = \frac{2-n}{2}$, which is
\begin{align*}
v_t = v_{rr} + \frac{1}{r}v_r - \left(\frac{n-2}{2}
\right)^2\frac{v}{r^2}.
\end{align*}
Set $\nu = \frac{n-2}{2}$ and recall the Hankel transform $\hat v (\rho, t) = \int_{0}^\infty v(r, t)J_{\nu}(\rho r)rdr$ with $J_\nu$ being the Bessel function, then we have 
\begin{align*}
\partial_t \hat v(\rho, t) &= \int_{0}^\infty\left[v'' + \frac{1}{r}v' - \left(\frac{n-2}{2}
\right)^2\frac{v}{r^2}\right]J_{\nu}(\rho r)dr\\
& = \int_{0}^\infty v\left[\rho^2rJ_\nu'' + \rho J_
\nu' - \left(\frac{n-2}{2}
\right)^2\frac{J_\nu}{r}\right]dr\\
& = \int_{0}^\infty v\left[-\rho^2rJ_\nu(\rho r)\right]dr.
\end{align*}
Hence we have 
\begin{align*}
\partial_t \hat v(\rho, t) & = -\rho^2\hat v(\rho, t).
\end{align*}
Here we have used the formula for the Bessel function,
\begin{align*}
(\rho r)J_\nu'' + \rho r J_\nu' + \left[(\rho r)^2-\nu^2\right] = 0.
\end{align*}

For the initial data, 
\begin{align*}
u(r, 0) = u_0(r) = r^{\frac{2-n}{2}}v(r, 0),
\end{align*}
then we have 
\begin{align*}
\hat v(\rho, 0) = \widehat{r^{\frac{n-2}{2}}u_0} = \int_{0}^\infty s^{\frac{n-2}{2}}u_0(s)J_{\nu}(\rho s)sds,
\end{align*}
and 
\begin{align*}
\hat v(\rho, t) = \widehat{r^{\frac{n-2}{2}}u_0} e^{-\rho^2 t}.
\end{align*}

Now from Hankel’s repeated integral (which holds for $\nu \geq -\frac{1}{2}$ and $\int_{0}^\infty |v(\rho, t)|d\rho < +\infty$, see formulas (10.22.76) and (10.22.77) in \cite{NIST:DLMF}), we have 
\begin{align*}
v(t, t) & = \int_{0}^\infty \hat v(\rho, t)J_\nu(\rho r) \rho d\rho \\
& =  \int_{0}^\infty \left(\int_{0}^\infty s^{\frac{n-2}{2}}u_0(s)J_{\nu}(\rho s)sds\right)e^{-\rho^2 t}J_\nu(\rho r) \rho d\rho \\
& = \int_{0}^\infty s^{\frac{n-2}{2}}u_0(s)sds\int_{0}^\infty e^{-\rho^2 t}J_{\nu}(\rho s)J_{\nu}(\rho r)\rho d\rho\\
& = \int_{0}^\infty s^{\frac{n}{2}}u_0(s)\frac{1}{2t}e^{-\frac{r^2+s^2}{4t}}I_{\frac{n-2}{2}}\left(\frac{rs}{2t}\right)ds.
\end{align*}
Therefore, we have 
\begin{align*}
u(r, t) = \int_{0}^\infty \Gamma_{n}(r, s; t)u_{0}(s)ds
\end{align*}
with 
\begin{align*}
\Gamma_{n}(r, s; t) = r^{\frac{2-n}{2}}\frac{s^{\frac{n}{2}}}{2t}e^{-\frac{r^2+s^2}{4t}}I_{\frac{n-2}{2}}\left(\frac{rs}{2t}\right).
\end{align*}
In the above, we have used the following formula,
\begin{align*}
\int_{0}^\infty e^{-\rho^2 t}J_{\frac{n-2}{2}}(\rho s)J_{\frac{n-2}{2}}(\rho r)\rho d\rho = \frac{1}{2t} e^{-\frac{r^2+s^2}{4t}}I_{\frac{n-2}{2}}\left(\frac{rs}{2t}\right).
\end{align*}
Here $I_{\frac{n-2}{2}}$ is the modified Bessel function of the first kind. Formula (10.22.67) in \cite{NIST:DLMF}, see also formula (6.633.2) in page 707 of \cite{Gradshteyn:1702455}.

\bigskip

\section*{Acknowledgements}
Y. Sire is partially supported by the NSF DMS Grant 2154219, ``Regularity vs singularity formation in elliptic and parabolic equations''. J.C. Wei is supported by National Key R\&D Program of China 2022YFA1005602, and Hong Kong General Research Fund “On critical and supercritical Fujita equation". Y. Zheng is supported by NSF of China (No. 12171355). Y. Zhou is supported in part by the Fundamental Research Funds for the Central Universities.

\bigskip


\begin{thebibliography}{10}

\bibitem{AbramowitzStegun}
Milton Abramowitz and Irene~A. Stegun.
\newblock {\em Handbook of mathematical functions with formulas, graphs, and
  mathematical tables}, volume No. 55 of {\em National Bureau of Standards
  Applied Mathematics Series}.
\newblock U. S. Government Printing Office, Washington, DC, 1964.

\bibitem{AHDM}
M.~F. Atiyah, N.~J. Hitchin, V.~G. Drinfeld, and Yu.~I. Manin.
\newblock Construction of instantons.
\newblock {\em Phys. Lett. A}, 65(3):185--187, 1978.

\bibitem{BPST}
A.~A. Belavin, A.~M. Polyakov, A.~S. Schwartz, and Yu.~S. Tyupkin.
\newblock Pseudoparticle solutions of the {Y}ang-{M}ills equations.
\newblock {\em Phys. Lett. B}, 59(1):85--87, 1975.

\bibitem{CortazarDelPinoMussoJEMS}
Carmen Cort\'azar, Manuel del Pino, and Monica Musso.
\newblock Green's function and infinite-time bubbling in the critical nonlinear
  heat equation.
\newblock {\em J. Eur. Math. Soc. (JEMS)}, 22(1):283--344, 2020.

\bibitem{DavilaDelPinoWei2020}
Juan D\'avila, Manuel del Pino, and Juncheng Wei.
\newblock Singularity formation for the two-dimensional harmonic map flow into
  {$S^2$}.
\newblock {\em Invent. Math.}, 219(2):345--466, 2020.

\bibitem{173D}
Manuel del Pino, Monica Musso, and Juncheng Wei.
\newblock Infinite-time blow-up for the 3-dimensional energy-critical heat
  equation.
\newblock {\em Anal. PDE}, 13(1):215--274, 2020.

\bibitem{NIST:DLMF}
{\it NIST Digital Library of Mathematical Functions}.
\newblock \url{https://dlmf.nist.gov/}, Release 1.2.4 of 2025-03-15.
\newblock F.~W.~J. Olver, A.~B. {Olde Daalhuis}, D.~W. Lozier, B.~I. Schneider,
  R.~F. Boisvert, C.~W. Clark, B.~R. Miller, B.~V. Saunders, H.~S. Cohl, and
  M.~A. McClain, eds.

\bibitem{DonaldsonKronheimer}
S.~K. Donaldson and P.~B. Kronheimer.
\newblock {\em The geometry of four-manifolds}.
\newblock Oxford Mathematical Monographs. The Clarendon Press, Oxford
  University Press, New York, 1990.
\newblock Oxford Science Publications.

\bibitem{Donaldson}
Simon Donaldson.
\newblock The {ADHM} construction of {Y}ang-{M}ills instantons.
\newblock {\em arXiv:2205.08639}.

\bibitem{feehan}
Paul~M.N. Feehan.
\newblock Global existence and convergence of solutions to gradient systems and
  applications to yang-mills gradient ow.
\newblock {\em arXiv preprint arXiv:1409.1525}, 2014.

\bibitem{FilaKing12}
Marek Fila and John~R. King.
\newblock Grow up and slow decay in the critical {S}obolev case.
\newblock {\em Netw. Heterog. Media}, 7(4):661--671, 2012.

\bibitem{Freedbook}
Daniel~S. Freed and Karen~K. Uhlenbeck.
\newblock {\em Instantons and four-manifolds}, volume~1 of {\em Mathematical
  Sciences Research Institute Publications}.
\newblock Springer-Verlag, New York, 1984.

\bibitem{Gradshteyn:1702455}
I.~S. Gradshteyn and I.~M. Ryzhik.
\newblock {\em Table of integrals, series, and products}.
\newblock Elsevier/Academic Press, Amsterdam, seventh edition, 2007.
\newblock Translated from the Russian, Translation edited and with a preface by
  Alan Jeffrey and Daniel Zwillinger.

\bibitem{Grotowski-Shatah}
Joseph~F. Grotowski and Jalal Shatah.
\newblock Geometric evolution equations in critical dimensions.
\newblock {\em Calc. Var. Partial Differential Equations}, 30(4):499--512,
  2007.

\bibitem{GNT10CMP}
Stephen Gustafson, Kenji Nakanishi, and Tai-Peng Tsai.
\newblock Asymptotic stability, concentration, and oscillation in harmonic map
  heat-flow, {L}andau-{L}ifshitz, and {S}chr\"{o}dinger maps on {$\Bbb R^2$}.
\newblock {\em Comm. Math. Phys.}, 300(1):205--242, 2010.

\bibitem{HaradaOsc}
Junichi Harada.
\newblock Oscillatory behavior of solutions to the critical fujita equation in
  6{D}.
\newblock {\em arXiv preprint arXiv:2511.17891}, 2025.

\bibitem{Hong-TianCAG}
Min-Chun Hong and Gang Tian.
\newblock Global existence of the {$m$}-equivariant {Y}ang-{M}ills flow in four
  dimensional spaces.
\newblock {\em Comm. Anal. Geom.}, 12(1-2):183--211, 2004.

\bibitem{LiWeiZhangZhou}
Zaizheng Li, Juncheng Wei, Qidi Zhang, and Yifu Zhou.
\newblock Long-time dynamics for the energy critical heat equation in
  {$\Bbb{R}^5$}.
\newblock {\em Nonlinear Anal.}, 247:Paper No. 113594, 15, 2024.

\bibitem{MeyerRiviereRMI}
Yves Meyer and Tristan Rivi\`ere.
\newblock A partial regularity result for a class of stationary {Y}ang-{M}ills
  fields in high dimension.
\newblock {\em Rev. Mat. Iberoamericana}, 19(1):195--219, 2003.

\bibitem{NaberValtortaInvenMath}
Aaron Naber and Daniele Valtorta.
\newblock Energy identity for stationary {Y}ang {M}ills.
\newblock {\em Invent. Math.}, 216(3):847--925, 2019.

\bibitem{naito}
Hisashi Naito.
\newblock Finite time blowing-up for the {Y}ang-{M}ills gradient flow in higher
  dimensions.
\newblock {\em Hokkaido Math. J.}, 23(3):451--464, 1994.

\bibitem{SchlatterCrelle}
Andreas Schlatter.
\newblock Global existence of the {Y}ang-{M}ills flow in four dimensions.
\newblock {\em J. Reine Angew. Math.}, 479:133--148, 1996.

\bibitem{Schlatter2}
Andreas Schlatter.
\newblock Long-time behaviour of the {Y}ang-{M}ills flow in four dimensions.
\newblock {\em Ann. Global Anal. Geom.}, 15(1):1--25, 1997.

\bibitem{SchlatterStruweAJM}
Andreas~E. Schlatter, Michael Struwe, and A.~Shadi Tahvildar-Zadeh.
\newblock Global existence of the equivariant {Y}ang-{M}ills heat flow in four
  space dimensions.
\newblock {\em Amer. J. Math.}, 120(1):117--128, 1998.

\bibitem{SWZ-YMHF}
Yannick Sire, Juncheng Wei, and Youquan Zheng.
\newblock Infinite time bubbling for the $ {SU} (2) $ {Y}ang-{M}ills heat flow
  on $\mathbb{R}^4$.
\newblock {\em arXiv preprint arXiv:2208.13875}, 2022.

\bibitem{StruweCVPDE}
Michael Struwe.
\newblock The {Y}ang-{M}ills flow in four dimensions.
\newblock {\em Calc. Var. Partial Differential Equations}, 2(2):123--150, 1994.

\bibitem{Tao-TianJAMS}
Terence Tao and Gang Tian.
\newblock A singularity removal theorem for {Y}ang-{M}ills fields in higher
  dimensions.
\newblock {\em J. Amer. Math. Soc.}, 17(3):557--593, 2004.

\bibitem{TianGangGTCG}
Gang Tian.
\newblock Gauge theory and calibrated geometry. {I}.
\newblock {\em Ann. of Math. (2)}, 151(1):193--268, 2000.

\bibitem{UY}
K.~Uhlenbeck and S.-T. Yau.
\newblock On the existence of {H}ermitian-{Y}ang-{M}ills connections in stable
  vector bundles.
\newblock volume~39, pages S257--S293. 1986.
\newblock Frontiers of the mathematical sciences: 1985 (New York, 1985).

\bibitem{UhlenbeckCMP-I}
Karen~K. Uhlenbeck.
\newblock Connections with {$L\sp{p}$}\ bounds on curvature.
\newblock {\em Comm. Math. Phys.}, 83(1):31--42, 1982.

\bibitem{UhlenbeckCMP-II}
Karen~K. Uhlenbeck.
\newblock Removable singularities in {Y}ang-{M}ills fields.
\newblock {\em Comm. Math. Phys.}, 83(1):11--29, 1982.

\bibitem{WaldronInventionMath}
Alex Waldron.
\newblock Long-time existence for {Y}ang-{M}ills flow.
\newblock {\em Invent. Math.}, 217(3):1069--1147, 2019.

\bibitem{4DFila-King}
Juncheng Wei, Qidi Zhang, and Yifu Zhou.
\newblock On {F}ila-{K}ing conjecture in dimension four.
\newblock {\em J. Differential Equations}, 398:38--140, 2024.

\bibitem{weizhouEPE}
Juncheng Wei and Yifu Zhou.
\newblock Some global solutions to the energy-critical semilinear heat
  equation.
\newblock {\em J. Elliptic Parabol. Equ.}, 11(3):2279--2301, 2025.

\end{thebibliography}

\bibliographystyle{plain}

\end{document}